\documentclass[a4paper,12pt,reqno]{amsart}
\usepackage[T1]{fontenc}
\usepackage[english]{babel}
\usepackage{amsmath,amsthm,amssymb}
\usepackage{fullpage}
\usepackage{ifthen}
\usepackage{color}
\usepackage{boxedminipage}
\usepackage{hyperref}

\usepackage{pgfplots}
\usepackage{pgfplotstable}
\pgfplotsset{
compat = 1.13,
tick label style = {font = \scriptsize},
legend style = {font = \scriptsize}
}
\usepackage{siunitx}
\usepackage{subcaption}
\makeatletter
\def\@seccntformat#1{%
  \protect\textup{\protect\@secnumfont
    \ifnum\pdfstrcmp{subsection}{#1}=0 \bfseries\fi
    \csname the#1\endcsname
    \protect\@secnumpunct
  }%
}
\makeatother
\usepackage{fancyhdr}
\advance\footskip0.4cm
\advance\textheight-0.4cm
\calclayout
\newtheorem{theorem}{Theorem}[section]

\newtheorem{lemma}[theorem]{Lemma}
\newtheorem{algorithm}[theorem]{Algorithm}
\newtheorem{definition}[theorem]{Definition}
\newtheorem{remark}[theorem]{Remark}

\newcommand\e{\boldsymbol{e}}
\newcommand{\m}{\boldsymbol{m}}
\newcommand{\n}{{\boldsymbol{n}}}
\newcommand\q{\boldsymbol{q}}
\renewcommand\u{\boldsymbol{u}}
\renewcommand\v{\boldsymbol{v}}
\newcommand\w{\boldsymbol{w}}
\newcommand\x{\boldsymbol{x}}
\newcommand\z{\boldsymbol{z}}
\newcommand\llambda{\boldsymbol{\lambda}}
\newcommand\mmu{\boldsymbol{\mu}}
\renewcommand\L{\boldsymbol{L}}
\def\C{\boldsymbol{C}}
\def\H{\boldsymbol{H}}
\def\ppi{\boldsymbol{\pi}}
\def\vvarphi{\boldsymbol{\varphi}}
\def\0{\boldsymbol{0}}
\newcommand{\abs}[1]{\left\lvert #1 \right\rvert}
\newcommand{\dt}{\mathrm{d}t}
\newcommand{\dx}{\mathrm{d}\boldsymbol{x}}
\newcommand{\EE}{\mathcal{E}}
\newcommand\eps{\varepsilon}
\newcommand{\eeta}{\boldsymbol{\eta}}
\newcommand{\grad}{\nabla}
\newcommand{\Grad}{\boldsymbol{\nabla}}
\newcommand{\heff}{\boldsymbol{h}_{\textrm{eff}}}
\newcommand{\heffhimex}{\boldsymbol{h}_{\textrm{eff}, h}^{\textrm{imex}}}
\newcommand{\hstray}{\boldsymbol{h}_{\textrm{s}}}
\newcommand{\inner}[3][]{\langle #2,#3 \rangle_{#1}}
\newcommand{\interp}{\mathcal{I}_h}
\newcommand{\Interp}{\boldsymbol{\mathcal{I}}_h}
\newcommand{\Kh}{\boldsymbol{\mathcal{K}}_h}
\newcommand{\Lapl}{\boldsymbol{\Delta}}
\newcommand{\RLapl}{\widetilde{\boldsymbol{\Delta}}}
\newcommand{\PPast}{\boldsymbol{P}_{\!\!\ast}}
\newcommand{\II}{\boldsymbol{I}}
\newcommand{\lex}{\ell_{\mathrm{ex}}}
\newcommand{\Mh}{\boldsymbol{\mathcal{M}}_h}
\newcommand{\mt}{\partial_t \boldsymbol{m}}
\newcommand{\N}{\mathbb{N}}
\newcommand{\norm}[2][]{\left\lVert #2 \right\rVert_{#1}}
\newcommand{\bignorm}[2][]{\big\lVert #2 \big\rVert_{#1}}
\renewcommand{\P}{\mathbb{P}}
\newcommand{\R}{\mathbb{R}}
\newcommand{\sphere}{\mathbb{S}^2}
\newcommand{\TT}{\mathcal{T}}
\newcommand{\NN}{\mathcal{N}}
\renewcommand\SS{\mathcal{S}}
\def\a{\boldsymbol{a}}
\def\b{\boldsymbol{b}}
\def\c{\boldsymbol{c}}
\def\d{\boldsymbol{d}}
\def\f{\boldsymbol{f}}
\def\g{\boldsymbol{g}}
\newcommand{\Range}{\mathcal{R}}
\newcommand{\Kernel}{\mathcal{N}}
\newcommand{\meas}{\operatorname{meas}}
\title{Unconditional well-posedness and IMEX improvement of a family of predictor-corrector methods in micromagnetics}
\author{Norbert J.\ Mauser}
\author{Carl-Martin Pfeiler}
\author{Dirk Praetorius}
\author{Michele Ruggeri}
\address{Institute of Analysis and Scientific Computing,
TU Wien,
Wiedner Hauptstrasse 8--10,
1040 Vienna, Austria}
\email{carl-martin.pfeiler@asc.tuwien.ac.at\quad\rm(corresponding author)}
\email{dirk.praetorius@asc.tuwien.ac.at}
\email{michele.ruggeri@asc.tuwien.ac.at}
\address{Research Platform MMM ''Mathematics-Magnetism-Materials'' c/o Faculty of Mathematics,
University of Vienna,
Oskar-Morgenstern-Platz 1,
1090 Vienna, Austria}
\email{norbert.mauser@univie.ac.at}
\date{\today}
\thanks{{\bf Acknowledgement.}
The authors thankfully acknowledge support by the Austrian Science Fund (FWF) through
the doctoral school \emph{Dissipation and dispersion in nonlinear PDEs} (grant W1245),
the special research program \emph{Taming complexity in partial differential systems} (grant F65),
and the project \emph{Reduced order approaches for micromagnetics} (grant P31140);
by the Vienna Science and Technology Fund (WWTF) through
the project \emph{Schr\"odinger Equations for QUantum EXperiments (SEQUEX)} (grant MA16-066);
and by the University of Vienna research platform MMM (''Mathematics-Magnetism-Materials'').
Further, we thank Lukas Exl for fruitful discussions in the early stage of this work.
}
\begin{document}
\begin{abstract}
Recently,
\emph{Kim \& Wilkening (Convergence of a mass-lumped finite element method for the Landau--Lifshitz equation, Quart.\ Appl.\ Math., 76, 383--405, 2018)} proposed two novel predictor-corrector methods
for the Landau--Lifshitz--Gilbert equation (LLG) in micromagnetics,
which models the dynamics of the magnetization in ferromagnetic materials.
Both integrators are based on the so-called Landau--Lifshitz form of LLG,
use mass-lumped variational formulations discretized by first-order finite elements,
and only require the solution of linear systems, despite the nonlinearity of LLG.
The first(-order in time) method combines a linear update
with an explicit projection of an intermediate approximation onto the unit sphere
in order to fulfill the LLG-inherent unit-length constraint at the discrete level.
In the second(-order in time) integrator, the projection step
is replaced by a linear constraint-preserving variational formulation.
In this paper,
we extend the analysis of the integrators
by proving unconditional well-posedness
and by establishing
a close connection of the methods with other approaches available in the literature.
Moreover, the new analysis also provides a well-posed integrator for the Schr\"odinger map equation (which is the limit case of LLG for vanishing damping).
Finally,
we design an implicit-explicit strategy
for the treatment of the lower-order field contributions,
which significantly reduces the computational cost of the schemes,
while preserving their theoretical properties.
\end{abstract}
\maketitle

\section{Introduction}

\subsection{Dynamic micromagnetism}
Reliable numerical simulations of magnetic processes
occurring at submicrometer length scales 
are fundamental tools
to optimize the design of many technological devices,
e.g., magnetic sensors, magnetic logic gates, and hard disk drives.
The theoretical background of most simulation packages
is the theory of \emph{micromagnetism}~\cite{brown1963},
a continuum theory
which models the magnetic state of a ferromagnetic material at constant temperature
in terms of a continuous vector field with constant magnitude, the \emph{magnetization}.
A well-accepted model to describe the dynamics of the magnetization
is a nonlinear parabolic partial differential equation (PDE)
usually referred to as
\emph{Landau--Lifshitz--Gilbert equation (LLG)} \cite{ll1935,gilbert2004},
which in the so-called \emph{Landau--Lifshitz (LL) form} reads as
\begin{equation} \label{eq:intro-ll}
 \mt = - \frac{1}{1 + \alpha^2} \, \m \times \heff(\m) - \frac{\alpha}{1 + \alpha^2} \, \m \times (\m \times \heff(\m)).
\end{equation}
Here,
$\m$ denotes the normalized magnetization,
which satisfies the nonconvex unit-length constraint $\abs{\m}=1$,
$\heff(\m)$ is the effective field, whose specific expression
depends on the Gibbs free energy of the system (see~\eqref{eq:gibbs} below),
and $\alpha \ge 0$ is the Gilbert damping parameter,
which incorporates energy dissipation into the model.

Alternative forms of LLG used in the literature,
mathematically equivalent to the LL form~\eqref{eq:intro-ll},
are
the so-called \emph{Gilbert form} of LLG
\begin{equation} \label{eq:llg:gilbert}
\mt = - \m \times \heff(\m) + \alpha \, \m \times \mt\,,
\end{equation}
and
\begin{equation}\label{eq:llg:alt}
\alpha \, \mt + \m \times \mt = -\m \times (\m \times \heff(\m)),
\end{equation}
which we call the \emph{alternative form} of LLG.

The aforementioned need of fast and reliable tools
to perform micromagnetic simulations
encouraged many works concerned with the
numerical analysis of LLG,
which will also be the subject of the present paper.

\subsection{State of the art}
In the last three decades,
mathematical questions arising from the micromagnetic theory
have been the subject of several studies, from both the analytical and the
numerical point of view.
For analytical results for LLG,
we refer, e.g., to the papers~\cite{visintin1985,as1992,gh1993,cf2001,melcher2005,ds2014,ft2017b,dip2020}
and the references therein.
For an overview of numerical methods proposed for LLG (up to 2008),
we refer
to the monograph~\cite{prohl2001}
and the review articles~\cite{kp2006,garciaCervera2007,cimrak2008b}.
More recently,
several numerical schemes with a rigorous convergence analysis
have been proposed.
They differ from each other in 
the LLG formulation they are based on (usually one among~\eqref{eq:intro-ll}--\eqref{eq:llg:alt}),
in the approach used to impose the unit-length constraint at the discrete level,
and in the type of convergence result
(plain convergence towards a weak solution of LLG with minimal regularity
or convergence with rates towards a sufficiently regular strong solution).

Semi-implicit finite element methods based on (variants of) the LL form~\eqref{eq:intro-ll} of LLG
are proposed in~\cite{gao2014,an2016},
where
\textsl{a~priori} error estimates, which show their convergence towards
a smooth solution of LLG, are also established.

A class of methods referred to as
\emph{tangent plane schemes} or
\emph{projection methods}~\cite{aj2006,bkp2008,alouges2008a,bffgpprs2014,ahpprs2014,akst2014,ft2017,dpprs2017,afkl2021}
is based on a predictor-corrector approach:
At each time-step, first,
an update is computed
by solving a linear variational problem
posed in the discrete tangent space of the current magnetization;
second, the update is used to obtain the magnetization at the next time-step.
The methods proposed in~\cite{aj2006,bkp2008,alouges2008a,bffgpprs2014,ahpprs2014,ft2017}
are based on a variational formulation of~\eqref{eq:llg:alt} discretized by first-order finite elements
to compute an approximation of the linear velocity $\mt$.
The magnetization at the next time-step is then obtained via a first-order time-stepping.
To impose the unit-length constraint at the vertices of the underlying mesh,
the nodal values are projected onto the sphere
in~\cite{aj2006,bkp2008,alouges2008a,bffgpprs2014}.
The projection is omitted from the time-stepping in~\cite{ahpprs2014,ft2017}:
In this case, the approximations do not fulfill the constraint
(not even at the vertices of the mesh),
but this error can be controlled by the time-step size
(in particular, the constraint holds for the solution of LLG
towards which the finite ele\-ment approximation converges).
High-order extensions of the tangent plane approach have been proposed in~\cite{akst2014,dpprs2017,afkl2021}.
The main advantages of this class of methods are
that they do not require any time-step restriction for convergence
(\emph{unconditional} convergence)~\cite{alouges2008a,bffgpprs2014,ahpprs2014,akst2014,ft2017,dpprs2017}
and that,
despite the nonlinear nature of LLG,
only one linear system per time-step has to be solved.

A numerical scheme based on the Gilbert form~\eqref{eq:llg:gilbert} of LLG
is considered in~\cite{bp2006,prs2017}.
The method employs mass-lumped first-order finite elements
for the spatial discretization
and the second-order implicit midpoint rule for the time discretization.
The scheme is unconditionally convergent towards a weak solution of LLG,
but requires the solution of a nonlinear system of equations
per time-step.
A similar method, but based on the LL form~\eqref{eq:intro-ll} of LLG,
is proposed and analyzed in~\cite{cimrak2009}.
The latter approach is motivated by the interest in having an integrator
which is robust with respect to the limit cases of~\eqref{eq:intro-ll}
in which one of the two terms on the right-hand side tends to zero.
Indeed, in the case $\heff(\m) = \Lapl\m$,
neglecting the second (dissipative) term on the left-hand side of~\eqref{eq:intro-ll} ($\alpha \to 0$),
one obtains the so-called Schr\"odinger map equation~\cite{ssb1986},
whereas omitting the first (conservative) term,
one is led to the harmonic map heat flow~\cite{lw2008}.

The recent work~\cite{kw2018} proposes two predictor-corrector schemes for LLG
which aim to combine the features of some of the above integrators.
In the first scheme, \cite[Algorithm~1]{kw2018},
which we denote by \texttt{PC1} for the sake of brevity,
the predictor is based on the LL form~\eqref{eq:intro-ll} of LLG
(like the variational formulation used in~\cite{cimrak2009})
and employs mass-lumping for its discretization (like~\cite{bp2006,cimrak2009}).
However,
it only requires the solution of one linear system per time-step
and uses the nodal projection to impose the unit-length constraint
(like the method of~\cite{alouges2008a,bffgpprs2014}).
The second scheme, \cite[Algorithm~2]{kw2018}, which we refer to as \texttt{PC2},
uses the same predictor as \texttt{PC1},
but replaces the nodal projection step with a constraint-preserving mass-lumped
(as in~\cite{bp2006,cimrak2009}),
but linear (as in~\cite{alouges2008a,bffgpprs2014}), variational formulation.
In the paper, adapting the proof of~\cite{alouges2008a}, the authors show
convergence of the approximations generated by \texttt{PC1} towards a weak solution of LLG.
Moreover,
the expected convergence order in time of both methods 
(first-order for \texttt{PC1}, second-order for \texttt{PC2})
is empirically verified by means of numerical experiments in 2D.

Note that in the above discussion
we have restricted ourselves to methods employing the finite element method for the spatial discretization.
For other approaches based on finite differences,
we refer, e.g., to~\cite{wge2001,dsm2005,kl2017,xgwzc2020,cwx2021} and the references therein.

\subsection{Novelty of the present work}
In this work, we improve the theoretical understanding of the predictor-corrector methods proposed in~\cite{kw2018}.

First, we show that \texttt{PC1} is unconditionally well-posed, i.e., for each time-step, the variational problem to be solved admits a unique solution, which is left open in the original paper.
By closing this fundamental gap, we show that \texttt{PC1} is not only closely related to the first-order tangent plane scheme of~\cite{alouges2008a,bffgpprs2014}, but actually can even be interpreted as a slight modification of it, which explains why the convergence analysis of the two schemes is almost identical.
Furthermore, following~\cite{bffgpprs2014}, we propose an implicit-explicit (IMEX) version of \texttt{PC1}.
When considering magnetization dynamics involving the full effective field---more precisely, dynamics including the nonlocal stray field---the proposed adaptation is computationally much more attractive:
The IMEX version \texttt{PC1+IMEX} avoids the costly inner iteration in the solver of the original scheme, while preserving the experimental first-order accuracy of \texttt{PC1}, which we confirm by numerical studies in 3D.

Second, we consider the analysis of \texttt{PC2}.
While the conservation of the unit-length constraint at the vertices of the mesh in \texttt{PC1} is guaranteed (at machine precision) also in practical computations (since it is directly enforced in the method using the nodal projection), the one guaranteed by \texttt{PC2}, which follows from the variational formulation of the corrector, is lost in practice due to the inevitable use of inexact (iterative) solvers for the solution of the arising linear systems.
Hence, although the predictors of \texttt{PC1} and \texttt{PC2} coincide in theory, the well-posedness analysis of (the predictor of) \texttt{PC1} does not transfer to a practical version of \texttt{PC2}.
To cope with this problem, we establish a decomposition of the finite element space, which does not only allow us to prove unconditional well-posedness of the practical version of \texttt{PC2}, but also to extend the result, for both \texttt{PC1} and \texttt{PC2} (theoretical and practical), to the limit case $\alpha=0$ (Schr\"odinger map equation).
Moreover, following~\cite{prs2017,dpprs2017}, we adopt the IMEX treatment also for \texttt{PC2}.
In particular, in the presence of the nonlocal stray field, the proposed method \texttt{PC2+IMEX} is computationally much more attractive than its fully implicit counterpart \texttt{PC2}, while conserving the experimental second-order accuracy in time.
Again, these claims are confirmed in our numerical studies.
Stability and convergence of \texttt{PC2}, not addressed in~\cite{kw2018}, remain open also in our analysis and will be the subject of future research.
In this paper, we shed some light on this question by means of some surprising numerical experiments.

\subsection{Outline}
We conclude this section by collecting some general notation and basic vector identities used throughout the work (Section~\ref{sec:notation}). 
In Section~\ref{sec:problem_formulation}, we formulate the initial boundary value problem for LLG in which we are interested, we recall the notion of a weak solution and introduce the basic ingredients of the discretization.
Section~\ref{sec:pc1} is devoted to the first-order method:
After proving unconditional well-posedness of \texttt{PC1} in Section~\ref{sec:pc1:well_posedness}, we propose an IMEX adaptation (Section~\ref{sec:pc1:imex}) overcoming the inefficiency drawbacks of the original version, while preserving unconditional well-posedness, stability, convergence (Section~\ref{sec:pc1:stability}), and accuracy.
Section~\ref{sec:pc2} is devoted to the second-order method:
In Section~\ref{sec:pc2:well_posedness}, we first prove unconditional well-posedness of \texttt{PC2}.
Subsequently, in Section~\ref{sec:pc2:inexact}, we extend the unconditional well-posedness result to the more general formulation of the second-order algorithm, where discrete unit-length of the iterates is not assumed.
This covers, in particular, the practical version of the scheme incorporating the inevitable use of inexact (iterative) linear solvers.
Section~\ref{sec:pc2:imex} closes with a second-order accuracy preserving IMEX modification overcoming the inefficiency drawbacks of \texttt{PC2}.
Section~\ref{sec:numerics} provides numerical studies validating the applicability (Section~\ref{sec:mumag4}) and the expected accuracy (Section~\ref{sec:convOrder}) of the IMEX integrators proposed in this work.
Finally, in Section~\ref{sec:stability_pc2}, we numerically investigate the stability of \texttt{PC2}.

\subsection{General notation and vector identities}
\label{sec:notation}
Throughout this work,
we use the standard notation for Lebesgue, Sobolev, and Bochner spaces and norms.
Vector-valued functions are indicated by bold letters.
Bold letters are also used for vector-valued and matrix-valued function spaces, e.g., both $L^2(\Omega; \R^3)$ and $L^2(\Omega; \R^{3\times 3})$ are denoted by $\L^2(\Omega)$.
We denote by $\inner{\cdot}{\cdot}$ and $\norm{\cdot}$ the scalar product and the norm of $\L^2(\Omega)$, respectively, while $\abs{\cdot}$ denotes the Euclidean norm of a vector in $\R^3$ or the Frobenius norm of a matrix in $\R^{3\times 3}$.
To abbreviate notation in proofs, we write $A \lesssim B$ when $A \le cB$ for some generic constant $c > 0$, which is clear from the context and always independent of the discretization parameters.
For vector-valued functions $\f, \g \colon \Omega \to \R^3$ we use the notation
\begin{align*}
- \g \times \Grad\f := \Grad\f \times \g := (\partial_1\f \times \g, \partial_2\f \times \g, \partial_3\f \times \g)\colon \Omega \to \R^{3 \times 3}\,.
\end{align*}
We conclude this section by recalling five vector identities used regularly in this work
\begin{subequations} \label{eq:cross}
\begin{align}
\label{eq:cross:a} \a \times \b &= -\b \times \a, \\
\label{eq:cross:b} (\a \times \b) \cdot \a &= 0, \\
\label{eq:cross:c} \a \times(\b \times \c) &= (\a \cdot \c) \, \b - (\a \cdot \b) \, \c, \\
\label{eq:cross:d} (\a \times \b) \cdot \c &= \a \cdot (\b \times \c), \\
\label{eq:cross:e} (\a \times \b) \cdot (\c \times \d) &= (\a \cdot \c) \, (\b \cdot \d) - (\b \cdot \c) \, (\a \cdot \d),
\end{align}
\end{subequations}
which hold true for arbitrary $\a, \b, \c, \d \in \R^3$.

\section{Problem formulation}
\label{sec:problem_formulation}

\subsection{Landau--Lifshitz--Gilbert equation}
\label{sec:llg}
Given a bounded Lipschitz domain $\Omega \subset \R^3$ and $T > 0$, we define the space-time cylinder $\Omega_T := \Omega \times (0,T)$.
We consider the following initial boundary value problem
\begin{subequations} \label{eq:llg:ibvp}
\begin{alignat}{2}
\label{eq:llg:ll}
(1 + \alpha^2) \, \mt & = - \m \times \heff(\m) - \alpha \, \m \times (\m \times \heff(\m)) 
&\qquad& \text{in } \Omega_T,
\\
\label{eq:llg:bc}
\partial_\n \m &= \0 
&& \text{on } \partial\Omega \times (0,T),
\\
\label{eq:llg:ic}
\m(0) &= \m^0
&& \text{in } \Omega.
\end{alignat}
The unknown is the normalized magnetization $\m\colon \Omega_T \to \sphere = \{ \x \in \R^3 \colon \abs{\x} = 1 \}$.
In~\eqref{eq:llg:ll},
the effective field
\begin{equation}\label{eq:llg:heff}
\heff(\m) = \lex^2 \, \Lapl \m + \ppi(\m) + \f
\end{equation}
\end{subequations}
is the negative functional derivative of the Gibbs free energy
\begin{equation} \label{eq:gibbs}
\EE(\m)
= \frac{\lex^2}{2} \int_\Omega \abs{\Grad\m}^2 \dx
- \frac{1}{2} \int_\Omega \ppi(\m) \cdot \m \, \dx
- \int_\Omega \f \cdot \m \, \dx,
\end{equation}
where $\lex>0$ is the exchange length, $\ppi \colon \L^2(\Omega) \to \L^2(\Omega)$ is a linear, continuous, and self-adjoint operator which collects all lower-order contributions such as uniaxial magnetocrystalline anisotropy and the nonlocal stray field, and $\f\colon \Omega_T \to \R^3$ is the applied external field.
The equation is supplemented with homogeneous Neumann boundary conditions~\eqref{eq:llg:bc} and the initial condition~\eqref{eq:llg:ic}, where $\m^0\colon \Omega \to \sphere$ denotes a given initial state.
\par
Taking the scalar product of~\eqref{eq:llg:ll} with $\m$, \eqref{eq:cross:b} yields that $0 = \partial_t \m \cdot \m$ in $\Omega_T$.
Since $\abs{\m^0} = 1$ in $\Omega$ by assumption and $\partial_t (|\m|^2 / 2) = \partial_t \m \cdot \m = 0$, it follows that $\abs{\m} = 1$ in $\Omega_T$.
Moreover, any solution of~\eqref{eq:llg:ll} satisfies the energy law
\begin{equation} \label{eq:energy_law}
\frac{\mathrm{d}}{\dt} \EE(\m,\f)
= - \alpha \int_\Omega \abs{\mt}^2 \dx
- \int_\Omega \partial_t\f \cdot \m \, \dx.
\end{equation}
From this, we see that the Gilbert damping constant $\alpha$ modulates the dissipation of the system.
In particular, if $\alpha=0$ and $\f$ is constant in time, then the energy is conserved.
The PDE inherent constraint $\abs{\m}=1$ in $\Omega_T$ and the energy law~\eqref{eq:energy_law} should be satisfied (at the discrete level) by any feasible numerical method.

\subsection{Weak solution}
\label{sec:weak}
We recall the notion of a weak solution of~\eqref{eq:llg:ibvp}, which extends the one introduced in~\cite{as1992}.
\begin{definition} \label{def:weak}
Let $\m^0 \in \H^1(\Omega;\sphere)$ and $\f \in C^1([0,T];\L^2(\Omega))$.
A vector field $\m\colon\Omega_T \to \R$ is called a weak solution of~\eqref{eq:llg:ibvp}, if the following properties are satisfied:
\begin{itemize}
\item[\rm(i)] $\m\in \H^1(\Omega_T) \cap L^{\infty}(0,T;\H^1(\Omega))$ with $\abs{\m}=1$ a.e.\ in $\Omega_T$;
\item[\rm(ii)] $\m(0)=\m^0$ in the sense of traces;
\item[\rm(iii)] for all $\w\in\H^1(\Omega_T)$, it holds that
\begin{equation} \label{eq:weak:variational}
\begin{split}
& \int_0^T \inner{\mt(t)}{\w(t)} \, \dt
- \alpha \int_0^T \inner{\m(t)\times\mt(t)}{\w(t)} \, \dt \\
& \quad = \lex^2 \int_0^T \inner{\m(t)\times\Grad\m(t)}{\Grad\w(t)} \, \dt
- \int_0^T \inner{\m(t)\times\ppi(\m(t))}{\w(t)} \, \dt \\
& \qquad - \int_0^T \inner{\m(t)\times\f(t)}{\w(t)} \, \dt;
\end{split}
\end{equation}
\item[\rm(iv)] it holds that
\begin{equation} \label{eq:weak:energy}
\EE(\m(T))
+ \alpha \int_0^T \norm{\mt(t)}^2 \dt
+ \int_0^T \inner{\partial_t \f(t)}{\m(t)} \, \dt
\leq \EE(\m^0).
\end{equation}
\end{itemize}
\end{definition}
We note that~\eqref{eq:weak:variational} is a variational formulation in space-time of
the Gilbert form~\eqref{eq:llg:gilbert} of LLG,
and that~\eqref{eq:weak:energy} is a weaker version of the energy law~\eqref{eq:energy_law}.

\subsection{Discretization}
\label{sec:discretization}
For the temporal discretization, given $L \in \N$, we consider a partition $\{ t_\ell \}_{\ell=0,\dots,L}$ of the time interval $[0,T]$ with uniform time-step size $k := T/L > 0$, i.e., $t_\ell = \ell k$ for all $\ell = 0, \dots, L$.
Given a finite sequence of functions $\{ \u^\ell \}_{\ell=0,\dots,L}$, we define
\begin{equation*}
\u^{\ell+1/2} := \frac{\u^{\ell+1} + \u^{\ell}}{2}
\quad
\text{and}
\quad
d_t\u^{\ell+1} := \frac{\u^{\ell+1} - \u^{\ell}}{k}
\quad
\text{for all } \ell = 0, \dots L-1.
\end{equation*}
For the spatial discretization, we consider a regular tetrahedral triangulation $\TT_h$ of $\Omega$ with mesh size $h>0$.
We denote by $\NN_h$ the set of vertices of $\TT_h$ and by $\{ \phi_{\z} \}_{\z \in \NN_h}$ the classical nodal basis of the space $\SS^1(\TT_h)$ of $\TT_h$-piecewise linear and globally continuous discrete functions, i.e., $\phi_{\z}(\z') = \delta_{\z, \z'}$ for all $\z, \z' \in \NN_h$.
With $\{ \e_j \}_{j=1,2,3}$ the standard basis of $\R^3$, $\{ \phi_{\z}\e_j \}_{\z \in \NN_h, j=1, 2, 3}$ gives a basis of $\SS^1(\TT_h)^3$.
Note that $\SS^1(\TT_h)^3$ is a $3N$-dimensional space, with $N$ denoting the number of vertices in $\NN_h$.
We introduce the \emph{set of admissible discrete magnetizations}
\begin{equation*}
\Mh := \left\{\m_h \in \SS^1(\TT_h)^3\colon \abs{\m_h(\z)}=1 \text{ for all } \z \in \NN_h \right\}
\end{equation*}
and, for $\m_h \in \Mh$, the \emph{discrete tangent space} of $\m_h$
\begin{equation*}
\Kh[\m_h] := \left\{\vvarphi_h \in \SS^1(\TT_h)^3 \colon \m_h(\z) \cdot \vvarphi_h(\z) = 0 \text{ for all } \z \in \NN_h \right\}.
\end{equation*}
We consider the nodal interpolant $\interp\colon C^0(\overline{\Omega})\to\SS^1(\TT_h)$, which is defined by $\interp(v) = \sum_{\z \in \NN_h} v(\z) \phi_{\z}$ for all $v \in C^0(\overline{\Omega})$.
We denote the vector-valued realization of the nodal interpolant by $\Interp\colon\C^0(\overline{\Omega})\to\SS^1(\TT_h)^3$.
In $\C^0(\overline{\Omega})$, besides the standard $\L^2(\Omega)$-scalar product $\inner{\cdot}{\cdot}$, we consider the mass-lumped scalar product $\inner[h]{\cdot}{\cdot}$ defined by
\begin{equation*}
\inner[h]{\u}{\w}
= \int_\Omega \interp(\u \cdot \w) \, \dx
\quad \text{for all } \u, \w \in \C^0(\overline{\Omega}).
\end{equation*}
Using the definition of the nodal interpolant, we see that
\begin{equation} \label{eq:mass-lumping}
\inner[h]{\u}{\w}
= \sum_{\z \in \NN_h} \beta_{\z} \, \u(\z)\cdot\w(\z)
\quad \text{for all } \u, \w \in \C^0(\overline{\Omega}),
\end{equation}
where $\beta_{\z} := \int_\Omega \phi_{\z} \, \dx > 0$ for all $\z \in \NN_h$.
For discrete functions, the induced norm $\norm[h]{\cdot}$ is equivalent to the standard $\L^2(\Omega)$-norm; see~\cite[Lemma~3.9]{bartels2015}, i.e., it holds that
\begin{equation}  \label{eq:normEquivalence}
\norm{\w_h} \leq \norm[h]{\w_h} \leq \sqrt{5} \norm{\w_h}
\quad \text{for all } \w_h \in \SS^1(\TT_h)^3.
\end{equation}
We define the (negative) discrete Laplacian $-\Lapl_h \colon \H^1(\Omega) \to \SS^1(\TT_h)^3$ by
\begin{equation} \label{eq:discrete_laplacian}
- \inner[h]{\Lapl_h \w}{\w_h}
= \inner{\Grad \w}{\Grad \w_h}
\quad \text{for all } \w \in \H^1(\Omega) \text{ and } \w_h \in \SS^1(\TT_h)^3.
\end{equation}
Let $\w_h \in \SS^1(\TT_h)^3$.
With a double application of the classical inverse estimate and the norm equivalence~\eqref{eq:normEquivalence}, we see that
\begin{equation*}
\begin{split}
\norm[h]{\Lapl_h \w_h}^2
&= \inner[h]{\Lapl_h \w_h}{\Lapl_h \w_h}
\stackrel{\eqref{eq:discrete_laplacian}}{=} - \inner{\Grad \w_h}{\Grad \Lapl_h \w_h}
\leq \norm{\Grad \w_h} \norm{\Grad \Lapl_h \w_h}\\
&\leq C h^{-2} \norm[h]{\w_h} \norm[h]{\Lapl_h \w_h}.
\end{split}
\end{equation*}
This shows that
\begin{equation} \label{eq:stability_laplacian}
\norm[h]{\Lapl_h \w_h} \leq C h^{-2} \norm[h]{\w_h}
\quad \text{for all } \w_h \in \SS^1(\TT_h)^3,
\end{equation}
where $C>0$ depends only on the quasi-uniformity of the triangulation $\TT_h$.
Finally, we define the mapping $\P_h \colon \L^2(\Omega) \to \SS^1(\TT_h)^3$ by
\begin{equation} \label{eq:L2-to-h-mapping}
\inner[h]{\P_h \w}{\w_h}
= \inner{\w}{\w_h}
\quad \text{for all } \w \in \L^2(\Omega) \text{ and } \w_h \in \SS^1(\TT_h)^3.
\end{equation}
Using~\eqref{eq:mass-lumping}, it is easy to see that, for all $\w \in \L^2(\Omega)$ and all $\z \in \NN_h$, it holds that $(\P_h \w)(\z) = \beta_{\z}^{-1} \int_{\Omega} \w \phi_{\z} \, \dx$.
In particular, the computation of $\P_h \w$ does not require to solve any linear system.

\section{First-order predictor-corrector scheme}
\label{sec:pc1}
In this section, we discuss the first-order scheme proposed in~\cite{kw2018} and its connections with the integrators proposed in~\cite{bp2006} and~\cite{alouges2008a}.
Our contribution is twofold:
First, we prove unconditional well-posedness of the scheme, which fills a fundamental gap in the analysis of~\cite{kw2018}.
Second, we employ an explicit treatment of the (nonlocal) lower-order contributions to obtain a computationally superior IMEX version of the scheme, preserving (unconditional) convergence and experimental rates in time.
We first consider the method for the case $\heff(\m) = \lex^2 \Lapl \m$.
For the general case $\heff(\m) = \lex^2 \Lapl \m + \ppi(\m) + \f$, we refer to Section~\ref{sec:pc1:imex}.

\subsection{Variational formulation}
\label{sec:pc1:variational}
The following algorithm restates~\cite[Algorithm~1]{kw2018} written in terms of the discrete functions $\m_h^\ell, \v_h^\ell, \m_h^{\ell+1} \in \SS^1(\TT_h)^3$, where $\m_h^\ell \approx \m(t_\ell)$, $\v_h^\ell \approx \partial_t\m(t_\ell)$, and $\m_h^{\ell+1} \approx \m(t_{\ell+1})$.
In particular, the predictor~\eqref{eq:pc1:variational:predictor} of Algorithm~\ref{alg:pc1:variational} reformulates the $N$ equations in $\R^3$ of the predictor of~\cite[Algorithm~1]{kw2018} as an equivalent variational formulation for $\v_h^\ell$ in $\SS^1(\TT_h)^3$.
As for the tangent plane scheme~\cite{alouges2008a}, $\theta \in [0,1]$ is a parameter modulating the `degree of implicitness' of the scheme.
\begin{algorithm}[\texttt{PC1}, variational form] \label{alg:pc1:variational}
\textbf{Input:}
$\m_h^0 \in \Mh$. \\
\textbf{Loop:}
For all time-steps $\ell = 0, \dots, L-1$, iterate:
\begin{itemize}
\item[\rm(i)] Compute $\v_h^\ell \in \SS^1(\TT_h)^3$ such that, for all $\w_h \in \SS^1(\TT_h)^3$, it holds that
\begin{equation} \label{eq:pc1:variational:predictor}
\begin{split}
(1 + \alpha^2) \inner[h]{\v_h^\ell}{\w_h}
&= -\lex^2 \inner[h]{\m_h^\ell \times \Lapl_h (\m_h^\ell + \theta k \v_h^\ell)}{\w_h} \\
&\quad -\alpha\lex^2 \inner[h]{\m_h^\ell \times (\m_h^\ell \times \Lapl_h (\m_h^\ell + \theta k \v_h^\ell))}{\w_h}\,.
\end{split}
\end{equation}
\item[\rm(ii)] Define $\m_h^{\ell+1} \in \Mh$ by
\begin{equation} \label{eq:pc1:variational:corrector}
\m_h^{\ell+1}(\z) := \frac{\m_h^{\ell}(\z) + k \v_h^{\ell}(\z)}{\abs{\m_h^{\ell}(\z) + k \v_h^{\ell}(\z)}} \in \sphere
\quad \text{for all } \z \in \NN_h.
\end{equation}
\end{itemize}
\textbf{Output:}
Sequence of discrete functions $\left\{(\v_h^\ell,\m_h^{\ell+1})\right\}_{\ell= 0}^{L-1}$.
\end{algorithm}

\subsection{Unconditional well-posedness}
\label{sec:pc1:well_posedness}
The predictor~\eqref{eq:pc1:variational:predictor} can be written as: Find $\v_h^\ell \in \SS^1(\TT_h)^3$ such that
\begin{align*}
a_{\operatorname{pre}}[\m_h^\ell](\v_h^\ell, \w_h) = F_{\operatorname{pre}}[\m_h^\ell](\w_h) \qquad \text{for all } \w_h \in \SS^1(\TT_h)^3\,,
\end{align*}
with some linear form $F_{\operatorname{pre}}[\m_h^\ell]$ and the bilinear form $a_{\operatorname{pre}}[\m_h^\ell]$ on $\SS^1(\TT_h)^3$ reading
\begin{align*}
a_{\operatorname{pre}}[\m_h^\ell](\v_h^\ell, \w_h)
&:= (1 + \alpha^2)\inner[h]{\v_h^\ell}{\w_h} 
+ \lex^2 \theta k \inner[h]{\m_h^\ell \times \Lapl_h\v_h^\ell}{\w_h} \\
&\quad + \alpha \lex^2 \theta k \inner[h]{\m_h^\ell \times (\m_h^\ell \times \Lapl_h\v_h^\ell)}{\w_h} \,.
\end{align*}
From the boundedness of $\m_h^\ell$ in $\L^\infty(\Omega)$ guaranteed by the nodal projection~\eqref{eq:pc1:variational:corrector} and an inverse estimate on the discrete Laplacian~\eqref{eq:stability_laplacian} we have
\begin{align*}
a_{\operatorname{pre}}[\m_h^\ell](\w_h, \w_h) 
\ge (1 - C kh^{-2}) \norm[h]{\w_h}^2\,.
\end{align*}
Hence, assuming the CFL condition $k = o(h^2)$ implies the coercivity of $a_{\operatorname{pre}}[\m_h^\ell]$ for sufficiently small $h$ and $k$.
However, this undesirable restriction is a consequence of naively using the inverse estimate, and can be avoided.

For arbitrary $\alpha > 0$ the upcoming refined analysis allows to drop any CFL-type assumptions on the discretization parameters:
In Lemma~\ref{lemma:pc1:Mh_Kh}, we first collect two basic properties of Algorithm~\ref{alg:pc1:variational}, which turn out to be sufficient to prove unconditional well-posedness of the algorithm in Theorem~\ref{thm:pc1:well_posedness}; also see Remark~\ref{re:pc1:nature}.
\begin{lemma}\label{lemma:pc1:Mh_Kh}
Let $\m_h^\ell \in \Mh$.
Suppose that the solution $\v_h^\ell \in \SS^1(\TT_h)^3$ to~\eqref{eq:pc1:variational:predictor} exists.
Then, $\v_h^\ell \in \Kh[\m_h^\ell]$, and~\eqref{eq:pc1:variational:corrector} provides a well-defined $\m_h^{\ell+1} \in \Mh$.
\end{lemma}
\begin{proof}
For arbitrary $\z \in \NN_h$, with $\phi_{\z} \in \SS^1(\TT_h)$ denoting the hat function with $\phi_{\z}(\z') = \delta_{\z, \z'}$ for all $\z' \in \NN_h$, we choose $\w_h := \m_h^\ell(\z)\phi_{\z} \in \SS^1(\TT_h)^3$ in~\eqref{eq:pc1:variational:predictor} to see
\begin{align*}
\m_h^\ell(\z) \cdot \v_h^\ell(\z) 
\stackrel{\eqref{eq:mass-lumping}}{=} \beta_{\z}^{-1} \inner[h]{\v_h^\ell}{\m_h^\ell(\z)\phi_{\z}} 
\stackrel{\eqref{eq:pc1:variational:predictor}, \eqref{eq:cross:b}}{=} 0 \,.
\end{align*}
Hence, $\v_h^\ell \in \SS^1(\TT_h)^3$ belongs to $\Kh[\m_h^\ell]$.\\
Well-posedness of \eqref{eq:pc1:variational:corrector} follows immediately from $\v_h^\ell \in \Kh[\m_h^\ell]$ via
\begin{align*}
|\m_h^\ell(\z) + k\v_h^\ell(\z)|^2 = |\m_h^\ell(\z)|^2 + k^2|\v_h^\ell(\z)|^2 \ge |\m_h^\ell(\z)|^2 = 1 \quad \text{for all } \ell = 0, \dots, L-1\,.
\end{align*}
Consequently, for all $\z \in \NN_h$ the denominator in~\eqref{eq:pc1:variational:corrector} is bounded below by $|\m_h^\ell(\z)| = 1$ and the corrector step of Algorithm~\ref{alg:pc1:variational} is always well-posed.\\
The third claim $\m_h^{\ell+1} \in \Mh$ follows directly from the explicit projection in~\eqref{eq:pc1:variational:corrector}.
\end{proof}
These two observations are already sufficient to prove the first main contribution of this work.
\begin{theorem}\label{thm:pc1:well_posedness}
Let $\alpha > 0$. Then, Algorithm~\ref{alg:pc1:variational} is unconditionally well-posed for any input $\m_h^0 \in \Mh$, i.e., for all $\ell = 0, \dots, L-1$ the predictor~\eqref{eq:pc1:variational:predictor} admits a unique solution $\v_h^\ell \in \SS^1(\TT_h)^3$ and the corrector~\eqref{eq:pc1:variational:corrector} is well-posed providing $\m_h^{\ell+1} \in \Mh$.
\end{theorem}
\begin{proof}
Well-posedness of the corrector~\eqref{eq:pc1:variational:corrector} and $\m_h^{\ell+1} \in \Mh$ follow from Lemma~\ref{lemma:pc1:Mh_Kh}.
Transforming~\eqref{eq:pc1:variational:predictor} into a coercive system in the discrete tangent space, we prove well-posedness of the predictor in three steps:
\begin{itemize}
\item \textbf{Step~1:}
The predictor of Algorithm~\ref{alg:pc1:variational} can be reformulated as a well-posed system.
\end{itemize}
We claim that $\v_h^\ell \in \SS^1(\TT_h)^3$ satisfies~\eqref{eq:pc1:variational:predictor} for all $\w_h \in \SS^1(\TT_h)^3$, if and only if it satisfies $\v_h^\ell \in \Kh[\m_h^\ell]$ as well as
\begin{equation} \label{eq:pc1:tpsh}
\alpha \inner[h]{\v_h^\ell}{\vvarphi_h}
+ \inner[h]{\m_h^\ell \times \v_h^\ell}{\vvarphi_h}
= \lex^2 \inner[h]{\Lapl_h (\m_h^\ell + \theta k \v_h^\ell)}{\vvarphi_h}
\quad \text{for all } \vvarphi_h \in \Kh[\m_h^\ell]\,.
\end{equation}
This formulation can be written as follows:
Find $\v_h^{\ell} \in \Kh[\m_h^\ell]$ such that
\begin{equation*}
a_{\operatorname{alt}}[\m_h^{\ell}]( \v_h^{\ell}, \vvarphi_h )
= \lex^2 \inner[h]{\Lapl_h \m_h^\ell}{\vvarphi_h}
\quad \text{for all } \vvarphi_h \in \Kh[\m_h^\ell],
\end{equation*}
where the bilinear form $a_{\operatorname{alt}}[\m_h^\ell] \colon \Kh[\m_h^\ell] \times \Kh[\m_h^\ell] \to \R$ is defined by
\begin{equation*}
a_{\operatorname{alt}}[ \m_h^\ell]( \v_h^{\ell}, \vvarphi_h )
:= 
\alpha \inner[h]{\v_h^\ell}{\vvarphi_h}
+ \inner[h]{\m_h^\ell \times \v_h^\ell}{\vvarphi_h}
- \lex^2 \theta k \inner[h]{\Lapl_h \v_h^\ell}{\vvarphi_h}.
\end{equation*}
For $\alpha > 0$, the bilinear form satisfies the ellipticity property
\begin{equation*}
a_{\operatorname{alt}}[ \m_h^\ell]( \vvarphi_h, \vvarphi_h )
= 
\alpha \norm[h]{\vvarphi_h}^2
+ \lex^2 \theta k \norm{\Grad \vvarphi_h}^2
\quad
\text{for all } \vvarphi_h \in \Kh[\m_h^\ell]\,,
\end{equation*}
and the problem~\eqref{eq:pc1:tpsh} is well-posed by the Lax--Milgram theorem.
To conclude the proof, it remains to show the claimed equivalence of~\eqref{eq:pc1:variational:predictor} and~\eqref{eq:pc1:tpsh}.
\begin{itemize}
\item \textbf{Step~2:}
Any solution $\v_h^\ell \in \SS^1(\TT_h)^3$ of~\eqref{eq:pc1:variational:predictor} also solves~\eqref{eq:pc1:tpsh}.
\end{itemize}
Given arbitrary $\vvarphi_h \in \Kh[\m_h^\ell]$, we choose $\w_h = \Interp(\alpha\vvarphi_h + \vvarphi_h \times \m_h^\ell) \in \SS^1(\TT_h)^3$ in~\eqref{eq:pc1:variational:predictor} to obtain
\begin{align}\label{eq:from_kw_predictor_to_tpsh}
\notag&(1+\alpha^2)\alpha\inner[h]{\v_h^\ell}{\vvarphi_h} 
+ (1+\alpha^2)\inner[h]{\v_h^\ell}{\vvarphi_h \times \m_h^\ell}
= - \alpha\lex^2 \inner[h]{\m_h^\ell \times \Lapl_h (\m_h^\ell + \theta k \v_h^\ell)}{\vvarphi_h} \\
\notag&\quad - \lex^2 \inner[h]{\m_h^\ell \times \Lapl_h (\m_h^\ell + \theta k \v_h^\ell)}{\vvarphi_h \times \m_h^\ell}
 -\alpha^2\lex^2 \inner[h]{\m_h^\ell \times (\m_h^\ell \times \Lapl_h (\m_h^\ell + \theta k \v_h^\ell))}{\vvarphi_h}\\
&\quad\phantom{=}- \alpha\lex^2 \inner[h]{\m_h^\ell \times (\m_h^\ell \times \Lapl_h (\m_h^\ell + \theta k \v_h^\ell))}{\vvarphi_h \times \m_h^\ell}\,.
\end{align}
By~\eqref{eq:cross:d} the left-hand side of~\eqref{eq:from_kw_predictor_to_tpsh} resembles the left-hand side of~\eqref{eq:pc1:tpsh} scaled by $(1+\alpha^2)$.
From $\m_h^{\ell} \in \Mh$ and $\vvarphi_h \in \Kh[\m_h^\ell]$, we infer $\interp(|\m_h^\ell|^2) = 1$ and $\interp(\m_h^\ell\cdot\vvarphi_h) = 0$ in $\Omega$.
Hence, using the vector identities~\eqref{eq:cross:b}--\eqref{eq:cross:e}, the first and the last term on the right-hand side of~\eqref{eq:from_kw_predictor_to_tpsh} cancel out, and~\eqref{eq:from_kw_predictor_to_tpsh} equivalently reads
\begin{align*}
(1 + \alpha^2)
\big(\alpha \inner[h]{\v_h^\ell}{\vvarphi_h}
+ \inner[h]{\m_h^\ell \times \v_h^\ell}{\vvarphi_h}\big)
 = (1+\alpha^2)\lex^2 \inner[h]{\Lapl_h (\m_h^\ell + \theta k \v_h^\ell)}{\vvarphi_h} \,.
\end{align*}
Now multiplying~\eqref{eq:from_kw_predictor_to_tpsh} by $1 / (1+\alpha^2)$, we conclude that any $\v_h^\ell \in \SS^1(\TT_h)^3$ satisfying~\eqref{eq:pc1:variational:predictor} necessarily satisfies~\eqref{eq:pc1:tpsh} and, according to Lemma~\ref{lemma:pc1:Mh_Kh}, belongs to $\Kh[\m_h^\ell]$ itself.
\begin{itemize}
\item \textbf{Step~3:}
Any solution $\v_h^\ell \in \Kh[\m_h^\ell]$ of~\eqref{eq:pc1:tpsh} also solves~\eqref{eq:pc1:variational:predictor}.
\end{itemize}
Given arbitrary $\w_h \in \SS^1(\TT_h)^3$, we choose $\vvarphi_h = \Interp\big(\m_h^\ell\times\w_h + \alpha\m_h^\ell \times (\w_h \times \m_h^\ell)\big) \in \Kh[\m_h^\ell]$ in~\eqref{eq:pc1:tpsh} to obtain
\begin{align}\label{eq:from_tpsh_to_kw_predictor}
\notag&\alpha\inner[h]{\v_h^\ell}{\m_h^\ell\times\w_h} 
+ \alpha^2\inner[h]{\v_h^\ell}{\m_h^\ell \times (\w_h \times \m_h^\ell)}\\
&\quad + \inner[h]{\m_h^\ell \times \v_h^\ell}{\m_h^\ell \times \w_h}
+ \alpha\inner[h]{\m_h^\ell \times \v_h^\ell}{\m_h^\ell \times (\w_h \times \m_h^\ell)} \\
\notag&\quad = \lex^2 \inner[h]{\Lapl_h (\m_h^\ell + \theta k \v_h^\ell)}{\m_h^\ell \times \w_h}
+ \alpha\lex^2 \inner[h]{\Lapl_h (\m_h^\ell + \theta k \v_h^\ell)}{\m_h^\ell \times (\w_h \times \m_h^\ell)}\,.
\end{align}
From $\m_h^{\ell} \in \Mh$ and $\v_h^\ell \in \Kh[\m_h^\ell]$, we infer $\interp(|\m_h^\ell|^2) = 1$ and $\interp(\m_h^\ell\cdot\v_h^\ell) = 0$ in $\Omega$.
Hence, by the vector identities~\eqref{eq:cross:b}--\eqref{eq:cross:e}, the first and the last term on the left-hand side of~\eqref{eq:from_tpsh_to_kw_predictor} cancel out, while the second and third term on the left-hand side of~\eqref{eq:from_tpsh_to_kw_predictor} add up to the left-hand side of~\eqref{eq:pc1:variational:predictor}.
Further, by~\eqref{eq:cross:d} the right-hand side of~\eqref{eq:from_tpsh_to_kw_predictor} resembles the right-hand side of~\eqref{eq:pc1:variational:predictor}.
We conclude that any $\v_h^\ell \in \Kh[\m_h^\ell] \subset \SS^1(\TT_h)^3$ satisfying~\eqref{eq:pc1:tpsh} necessarily satisfies~\eqref{eq:pc1:variational:predictor}.
Ultimately, we have shown that~\eqref{eq:pc1:variational:predictor} is equivalent to~\eqref{eq:pc1:tpsh}, which always allows for a unique solution as shown in Step~1.
\end{proof}
\begin{remark}\label{re:pc1:nature}
{\rm(i)}
Let $\w\colon \Omega \to \R^3$ be an arbitrary smooth test function.
Writing $\m^\ell := \m(t_\ell)$ and $\v^\ell := \partial_t \m(t_\ell)$, the variational formulation of the LL form~\eqref{eq:llg:ll} of LLG at time  $t_\ell \in (0, T)$ reads
\begin{equation*}
(1 + \alpha^2) \inner{\v^\ell}{\w}
\stackrel{\phantom{\eqref{eq:cross:a}}}{=} - \lex^2 \, \inner{\m^\ell \times \Lapl \m^\ell}{\w}
- \alpha \lex^2 \, \inner{\m^\ell \times (\m^\ell \times \Lapl \m^\ell)}{\w}\,.
\end{equation*}
The discrete variational formulation~\eqref{eq:pc1:variational:predictor} can be seen as a discrete mass-lumped version of the latter, where the effective field is treated implicitly in time.\\
{\rm(ii)} The core of the proof of Theorem~\ref{thm:pc1:well_posedness} is the equivalent reformulation of the predictor step~\eqref{eq:pc1:variational:predictor} as well-posed system~\eqref{eq:pc1:tpsh} in the discrete tangent space $\Kh[\m_h^\ell]$.
For $\alpha > 0$, the reformulated system is unconditionally well-posed and corresponds to a discretization of the alternative form of LLG \eqref{eq:llg:alt}.
Using~\eqref{eq:cross:c} and $|\m|^2 \equiv 1$, the formulation~\eqref{eq:llg:alt} is directly obtained from the LL form~\eqref{eq:llg:ll} via $(\alpha \cdot \eqref{eq:llg:ll} + \m \times \eqref{eq:llg:ll}) / (1 + \alpha^2)$.
Step~3 of the proof of Theorem~\ref{thm:pc1:well_posedness} resembles the analogous computations on a discrete level.
We emphasize, that the mass-lumped scalar product $\inner[h]{\cdot}{\cdot}$ as well as $\m_h^\ell \in \Mh$ and $\v_h^\ell \in \Kh[\m_h^\ell]$ are the crucial ingredients in the proof of Theorem~\ref{thm:pc1:well_posedness}.\\
{\rm(iii)} With the reformulation~\eqref{eq:pc1:tpsh}, we fully understand the real nature of the first-order integrator from~\cite{kw2018}:
It is a predictor-corrector scheme which combines the approaches of Bartels \& Prohl~\cite{bp2006} (mass-lumping~\eqref{eq:mass-lumping}, discrete Laplacian~\eqref{eq:discrete_laplacian}) and Alouges~\cite{alouges2008a} (degree of implicitness $\theta$, projection update~\eqref{eq:pc1:variational:corrector}, unknown approximates time derivative).
The predictor step~\eqref{eq:pc1:variational:predictor} is a mass-lumped discrete variational formulation of the LL form~\eqref{eq:llg:ll} of LLG.
The equivalent variational formulation~\eqref{eq:pc1:tpsh} is a mass-lumped variational formulation of the alternative form~\eqref{eq:llg:alt} of LLG and, in particular, is the mass-lumped version of the predictor step of the tangent plane scheme from~\cite{alouges2008a}.
Analogously to the tangent plane scheme, the corrector step of Algorithm~\ref{alg:pc1:variational} employs the nodal projection to enforce the modulus constraint at the vertices of the triangulations.\\
{\rm(iv)} While the proof of Theorem~\ref{thm:pc1:well_posedness} emphasizes the close relation of Algorithm~\ref{alg:pc1:variational} to the first-order tangent plane scheme, it is restricted to $\alpha > 0$.
In fact, Theorem~\ref{thm:pc1:well_posedness} can also be proved for the limit case $\alpha = 0$; see Remark~\ref{re:proof:inexact}{\rm(iii)}--{\rm(iv)} below.
\end{remark}

\subsection{Including lower-order contributions}
\label{sec:pc1:imex}
In this section, we discuss the extension of the scheme to the general case $\heff(\m) = \lex^2 \, \Lapl \m + \ppi(\m) + \f$.
We start by recalling the definition~\eqref{eq:L2-to-h-mapping} of the mapping $\P_h \colon \L^2(\Omega) \to \L^2(\Omega)$ and assume that we are given an operator $\ppi_h\colon\L^2(\Omega) \to \L^2(\Omega)$ which approximates $\ppi$, e.g., in the case of the nonlocal stray field $\ppi(\m) = \hstray$, $\ppi_h$ is a method for the approximation of the magnetostatic Maxwell equations, e.g., via the hybrid FEM-BEM method from~\cite{fk1990}.
\par
In the original first-order integrator from~\cite{kw2018}, the lower-order contributions are treated implicitly in time.
Rewritten as a mass-lumped discrete LL formulation like~\eqref{eq:pc1:variational:predictor}, the predictor step of~\cite[Algorithm~1]{kw2018} reads as follows:
Find $\v_h^\ell \in \SS^1(\TT_h)^3$ such that
\begin{align}\label{eq:pc1:implicit_pi}
\notag&(1 + \alpha^2) \inner[h]{\v_h^\ell}{\w_h}
= -\inner[h]{\m_h^\ell \times [\lex^2\Lapl_h (\m_h^\ell + \theta k \v_h^\ell) + \P_h(\ppi_h(\m_h^\ell + \theta k \v_h^\ell) + \f^{\ell + \theta})]}{\w_h} \\
&\qquad -\alpha\inner[h]{\m_h^\ell \times (\m_h^\ell \times [\lex^2\Lapl_h (\m_h^\ell + \theta k \v_h^\ell) + \P_h(\ppi_h(\m_h^\ell + \theta k \v_h^\ell) + \f^{\ell + \theta})])}{\w_h}
\end{align}
for all $\w_h \in \SS^1(\TT_h)^3$.
Here, $\f^{\ell+\theta} = \f(t_\ell + \theta k)$ for all $\ell = 0, \dots, L-1$.
However, this approach for the inclusion of the lower-order terms is not very attractive from the computational point of view:
Indeed, the variational formulation comprises the term $\ppi_h(\v_h^\ell)$ which requires to solve a (possibly nonlocal) problem for the unknown.
An implementation of this scheme would then be based on a costly inner iteration.
\par
From our previous work on the tangent plane scheme~\cite{bffgpprs2014,dpprs2017} and on the midpoint scheme~\cite{prs2017}, we know that an explicit treatment is favorable:
Therefore, we change the above variational formulation: Find $\v_h^\ell \in \SS^1(\TT_h)^3$ such that
\begin{equation*}
\begin{split}
(1 + \alpha^2) \inner[h]{\v_h^\ell}{\w_h}
&= -\inner[h]{\m_h^\ell \times [\lex^2\Lapl_h (\m_h^\ell + \theta k \v_h^\ell) + \P_h(\ppi_h(\m_h^\ell) + \f^{\ell})]}{\w_h} \\
&\quad -\alpha\inner[h]{\m_h^\ell \times (\m_h^\ell \times [\lex^2\Lapl_h (\m_h^\ell + \theta k \v_h^\ell) + \P_h(\ppi_h(\m_h^\ell) + \f^{\ell})])}{\w_h}
\end{split}
\end{equation*}
for all $\w_h \in \SS^1(\TT_h)^3$.
Only the leading-order exchange contribution is treated implicitly in time, while the lower-order contributions are treated explicitly.
This does not spoil the convergence result of the scheme (since the nodal projection already restricts the scheme to first-order in time) and it is computationally much more attractive.
To sum up, we consider the following implicit-explicit (IMEX) algorithm.
\begin{algorithm}[\texttt{PC1+IMEX}] \label{alg:pc1:imex}
\textbf{Input:}
$\m_h^0 \in \Mh$. \\
\textbf{Loop:}
For all time-steps $\ell = 0, \dots, L-1$, iterate:
\begin{itemize}
\item[\rm(i)] Compute $\P_h(\ppi_h(\m_h^\ell)) \in \SS^1(\TT_h)^3$.
\item[\rm(ii)] Compute $\v_h^\ell \in \SS^1(\TT_h)^3$ such that, for all $\w_h \in \SS^1(\TT_h)^3$, it holds that
\begin{align} \label{eq:pc1:imex:predictor}
(1 + \alpha^2) \inner[h]{\v_h^\ell}{\w_h}
&= -\inner[h]{\m_h^\ell \times [\lex^2\Lapl_h (\m_h^\ell + \theta k \v_h^\ell) + \P_h(\ppi_h(\m_h^\ell) + \f^{\ell})]}{\w_h} \\
\notag&\quad -\alpha\inner[h]{\m_h^\ell \times (\m_h^\ell \times [\lex^2\Lapl_h (\m_h^\ell + \theta k \v_h^\ell) + \P_h(\ppi_h(\m_h^\ell) + \f^{\ell})])}{\w_h}\,.
\end{align}
\item[\rm(iii)] Define $\m_h^{\ell+1} \in \Mh$ by
\begin{equation} \label{eq:pc1:imex:corrector}
\m_h^{\ell+1}(\z) := \frac{\m_h^{\ell}(\z) + k \v_h^{\ell}(\z)}{\abs{\m_h^{\ell}(\z) + k \v_h^{\ell}(\z)}} \in \sphere
\quad \text{for all } \z \in \NN_h.
\end{equation}
\end{itemize}
\textbf{Output:}
Sequence of discrete functions $\left\{(\v_h^\ell,\m_h^{\ell+1})\right\}_{\ell= 0}^{L-1}$.
\end{algorithm}

\subsection{Stability of Algorithm~\ref{alg:pc1:imex}}
\label{sec:pc1:stability}
Well-posedness of Algorithm~\ref{alg:pc1:imex} follows from well-posedness of Algorithm~\ref{alg:pc1:variational} (Theorem~\ref{thm:pc1:well_posedness}), as the system matrices for the linear systems corresponding to the left-hand sides of~\eqref{eq:pc1:imex:predictor} and~\eqref{eq:pc1:variational:predictor}, respectively, coincide.
\par
For stability of Algorithm~\ref{alg:pc1:imex}, we assume that all off-diagonal entries of the stiffness matrix $A=(a_{\z, \z'})_{\z, \z' \in \NN_h}$ are nonpositive, i.e., it holds that
\begin{equation} \label{eq:angleCondition}
a_{\z, \z'} = \inner{\grad\phi_{\z'}}{\grad\phi_{\z}} \leq 0
\quad \text{for all } \z, \z' \in \NN_h \text{ with } \z \neq \z'.
\end{equation}
This requirement, usually referred to as \emph{angle condition}\footnote{The assumption~\eqref{eq:angleCondition} is usually referred to as angle condition, because in 3D it is satisfied, e.g., if all dihedral angles of all tetrahedra of $\TT_h$ are $\leq \pi/2$.}, ensures that the nodal projection $\w_h \mapsto \Interp\big[\w_h/\abs{\w_h}\big]$ does not increase the exchange energy of a discrete function, i.e., it holds that
\begin{equation} \label{eq:nodalProjectionEnergy}
\norm{\Grad\Interp\big[\w_h/\abs{\w_h}\big]} \leq \norm{\Grad\w_h}\,,
\end{equation}
for all $\w_h \in \SS^1(\TT_h)^3$ with $\abs{\w_h(\z)}\geq 1$ for all $\z\in\NN_h$; see~\cite[Lemma~3.2]{bartels2005}.
Moreover, we assume that the discrete operator $\ppi_h\colon \SS^1(\TT_h)^3 \to \L^2(\Omega)$ is stable in the sense that
\begin{equation} \label{eq:pi:stability}
\norm{\ppi_h(\w_h)} \leq C \norm{\w_h}
\quad \text{for all } \w_h \in \SS^1(\TT_h)^3\,,
\end{equation}
which is met in many practical situations; see~\cite{bffgpprs2014}.
Under these assumptions, there holds stability of Algorithm~\ref{alg:pc1:imex}.
\begin{theorem}\label{thm:pc1:imex:stability}
Let $\TT_h$ such that~\eqref{eq:nodalProjectionEnergy} holds true.
For input $\m_h^0 \in \Mh$, let $\left\{(\v_h^\ell,\m_h^{\ell+1})\right\}_{\ell= 0}^{L-1}$ be the output of Algorithm~\ref{alg:pc1:imex}.
Then, for all $J = 0, \dots, L-1$, there holds the stability estimate
\begin{equation}\label{eq:pc1:imex:stability}
\begin{split}
& \frac{\lex^2}{2} \norm{\Grad\m_h^J}^2
+ \alpha k \sum_{\ell=0}^{J-1} \norm{\v_h^\ell}^2
+ \lex^2 (\theta-1/2) k^2 \sum_{\ell=0}^{J-1} \norm{\Grad\v_h^\ell}^2 \\
& \quad \leq \frac{\lex^2}{2} \norm{\Grad\m_h^0}^2
+ k \sum_{\ell=0}^{J-1} \inner{\v_h^\ell}{\ppi_h(\m_h^\ell) + \f^\ell}.
\end{split}
\end{equation}
\end{theorem}
\begin{proof}
To abbreviate notation we define
\begin{align*}
\heffhimex(\m_h^\ell, \v_h^\ell) := \lex^2 \Lapl_h (\m_h^\ell + \theta k \v_h^\ell) + \P_h(\ppi_h(\m_h^\ell) + \f^\ell) \in \SS^1(\TT_h)^3\,.
\end{align*}
Testing~\eqref{eq:pc1:imex:predictor} with $\w_h = \v_h^\ell$, $\w_h = \heffhimex(\m_h^\ell, \v_h^\ell)$, and $\w_h = \Interp(\m_h^\ell \times \heffhimex(\m_h^\ell, \v_h^\ell))$, respectively, leads to
\begin{subequations}
\begin{align}
\label{eq:pc1:imex:tested:a}(1+ \alpha^2)\norm[h]{\v_h^\ell}^2 
&= \inner[h]{\m_h^\ell \times \v_h^\ell}{\heffhimex(\m_h^\ell, \v_h^\ell)}
+ \alpha\inner[h]{\v_h^\ell}{\heffhimex(\m_h^\ell, \v_h^\ell)},\\
\label{eq:pc1:imex:tested:b}\alpha\norm[h]{\m_h^\ell \times \heffhimex(\m_h^\ell, \v_h^\ell)}^2 
&= (1+ \alpha^2)\inner[h]{\v_h^\ell}{\heffhimex(\m_h^\ell, \v_h^\ell)}\,, \\
\label{eq:pc1:imex:tested:c}\norm[h]{\m_h^\ell \times \heffhimex(\m_h^\ell, \v_h^\ell)}^2 
&= (1+ \alpha^2)\inner[h]{\m_h^\ell \times \v_h^\ell}{\heffhimex(\m_h^\ell, \v_h^\ell)}\,,
\end{align}
\end{subequations}
where we used $\interp(|\m_h^\ell|^2) = 1$ and $\interp(\m_h^\ell\cdot\v_h^\ell) = 0$ in $\Omega$ together with the identities~\eqref{eq:cross:b}--\eqref{eq:cross:e}.
Combining~\eqref{eq:pc1:imex:tested:a}--\eqref{eq:pc1:imex:tested:c} gives
\begin{align*}
\alpha\norm[h]{\v_h^\ell}^2 = \inner[h]{\v_h^\ell}{\heffhimex(\m_h^\ell, \v_h^\ell)}\,.
\end{align*}
Plugging in the definition of $\heffhimex(\m_h^\ell, \v_h^\ell)$, we see
\begin{equation}\label{eq:pc1:imex:expand:heffhimex}
\lex^2 \inner{\Grad \v_h^\ell}{\Grad\m_h^\ell}
= - \alpha \norm[h]{\v_h^\ell}^2
- \lex^2 \theta k \norm{\Grad \v_h^\ell}^2
+ \inner{\v_h^\ell}{\ppi_h(\m_h^\ell) + \f^\ell}.
\end{equation}
Using the angle condition, we deduce that
\begin{equation*}
\begin{split}
& \frac{\lex^2}{2} \norm{\Grad\m_h^{\ell+1}}^2
- \frac{\lex^2}{2} \norm{\Grad\m_h^\ell}^2
\stackrel{\eqref{eq:nodalProjectionEnergy}}{\leq} \frac{\lex^2}{2} \norm{\Grad(\m_h^\ell + k \v_h^\ell)}^2 
- \frac{\lex^2}{2} \norm{\Grad\m_h^\ell}^2 \\
& \quad = \lex^2 k \inner{\Grad\m_h^\ell}{\Grad \v_h^\ell}
+ \frac{\lex^2}{2} k^2 \norm{\Grad\v_h^\ell}^2 \\
& \quad \stackrel{\eqref{eq:pc1:imex:expand:heffhimex}}{=} - \alpha k \norm[h]{\v_h^\ell}^2
- \lex^2 (\theta-1/2) k^2 \norm{\Grad\v_h^\ell}^2
+ k \inner{\v_h^\ell}{\ppi_h(\m_h^\ell) + \f^\ell}.
\end{split}
\end{equation*}
Summing over $\ell = 0, \dots, J-1$, we obtain that
\begin{equation*}
\begin{split}
& \frac{\lex^2}{2} \norm{\Grad\m_h^J}^2
+ \alpha k \sum_{\ell=0}^{J-1} \norm[h]{\v_h^\ell}^2
+ \lex^2 (\theta-1/2) k^2 \sum_{\ell=0}^{J-1} \norm{\Grad\v_h^\ell}^2 \\
& \quad \leq \frac{\lex^2}{2} \norm{\Grad\m_h^0}^2
+ k \sum_{\ell=0}^{J-1} \inner{\v_h^\ell}{\ppi_h(\m_h^\ell) + \f^\ell}.
\end{split}
\end{equation*}
Finally, the norm equivalence~\eqref{eq:normEquivalence} yields~\eqref{eq:pc1:imex:stability}.
\end{proof}
\begin{remark}\label{re:pc1:convergence}
{\rm(i)} The stability~\eqref{eq:pc1:imex:stability} is the very same estimate that one obtains for the first-order tangent plane scheme from~\cite{alouges2008a}; see, e.g., \cite[Lemma~3.5]{bffgpprs2014}.
Combining this estimate with the stability of $\ppi_h$ from~\eqref{eq:pi:stability}, one obtains boundedness of the discrete solutions, which allows to apply the standard compactness argument for parabolic PDEs to prove convergence; see, e.g.,~\cite[Section~3]{alouges2008a} or \cite[Section~3.5]{bffgpprs2014}.\\
{\rm(ii)} Consequently, for both Algorithm~\ref{alg:pc1:variational} and Algorithm~\ref{alg:pc1:imex}, one obtains a convergence result identical to~\cite[Theorem~2, Remark~1]{alouges2008a}.
In particular, as $h, k \to 0$, for $1/2 < \theta \le 1$ no coupling of the discretization parameters is necessary, while the CFL conditions $k = o(h)$ and $k = o(h^2)$ are proved to be sufficient for $\theta = 1/2$ and $0 \le \theta < 1/2$, respectively.\\
{\rm(iii)} We note that~\cite[Theorem~2.2]{kw2018} and its proof are slightly inaccurate and, in particular, the CFL condition $k = o(h)$ is missing for $\theta = 1/2$.
\end{remark}
We briefly comment on a projection-free modification of \texttt{PC1+IMEX}.
\begin{remark}
As pointed out in Remark~\ref{re:pc1:nature}, Algorithm~\ref{alg:pc1:imex} and the first-order tangent-plane scheme from~\cite{alouges2008a} coincide up to mass-lumped integration in the predictor~\eqref{eq:pc1:imex:predictor}.
Hence, an obvious modification of Algorithm~\ref{alg:pc1:imex} in the spirit of the projection-free tangent-plane scheme from~\cite[Algorithm~6]{ahpprs2014} is omitting the projection in the corrector~\eqref{eq:pc1:imex:corrector}, i.e., defining $\m_h^{\ell+1} := \m_h^\ell + k \v_h^\ell \in \SS^1(\TT_h)^3$.
For this projection-free variant of Algorithm~\ref{alg:pc1:imex}, at first glance, one could hope for the same desirable theoretical features as for the projection-free tangent plane scheme --- namely stability and weak convergence~\cite{ahpprs2014} without the angle condition~\eqref{eq:angleCondition} and even strong convergence~\cite{ft2017}, both at the price of a slight deterioration from nodewise unit-length $\m_h^\ell \not\in\Mh$.
In contrast to the projection-free tangent plane scheme, the projection-free variant of Algorithm~\ref{alg:pc1:imex} is unconditionally well-posed even for the limit case $\alpha = 0$; see Remark~\ref{re:proof:inexact}{\rm(iii)}--{\rm(iv)} below.
Further, it satisfies a discrete energy law, which, e.g., in the exchange-only case for $\theta = 1/2$ reads
\begin{align*}
\frac{\lex^2}{2} \norm{\Grad\m_h^J}^2
+ \frac{\alpha}{1+\alpha^2}\lex^4 k \sum_{\ell=0}^{J-1} \norm[h]{\m_h^\ell \times \Lapl_h(\m_h^\ell + (k/2) \v_h^\ell)}^2
= \frac{\lex^2}{2} \norm{\Grad\m_h^0}^2\,.
\end{align*}
However, due to the loss of nodewise unit-length $\m_h^\ell \not\in \Mh$, equivalence of the predictor of the projection-free version of Algorithm~\ref{alg:pc1:imex} and the discrete tangent space system~\eqref{eq:pc1:tpsh} in $\Kh[\m_h^\ell]$ does not hold anymore.
Consequently, the analysis for the projection-free tangent plane scheme from~\cite{ahpprs2014, ft2017} does not (directly) transfer, and a rigorous analysis of the projection-free version of Algorithm~\ref{alg:pc1:imex} remains open.
\end{remark}

\section{Second-order predictor-corrector scheme}
\label{sec:pc2}
In this section, we discuss the second-order scheme proposed in~\cite{kw2018}.
Our contribution is threefold:
In theory, well-posedness (for the predictor) of the scheme (which was left open in~\cite{kw2018}) follows already from our analysis in Section~\ref{sec:pc1:well_posedness}.
When accounting for the use of inexact (iterative) linear solvers, which is inevitable in practice, however, discrete unit-length $\m_h^\ell \in \Mh$ is lost and therefore a conceptually new analysis is required to guarantee well-posedness in practice.
We fill this fundamental gap in the analysis of~\cite{kw2018} for their second-order scheme, by proving unconditional well-posedness not only for the proposed predictor-corrector scheme, but also for its practical version incorporating inexact (iterative) linear solvers.
Again, we first consider the method for the case $\heff(\m) = \lex^2 \Lapl \m$.
The general case $\heff(\m) = \lex^2 \Lapl \m + \ppi(\m) + \f$ is treated in Section~\ref{sec:pc2:imex}, where we employ an explicit treatment of the (nonlocal) lower-order contributions to obtain a computationally superior IMEX version of the scheme, preserving experimental rates in time.
We numerically confirm the applicability and the formal second-order of the proposed IMEX scheme in Section~\ref{sec:numerics}.
Theoretical stability (and hence convergence) of the second-order scheme remains open (like in~\cite{kw2018}), but is experimentally investigated in a numerical study in Section~\ref{sec:stability_pc2}.

\subsection{Variational formulation}
The following algorithm restates~\cite[Algorithm~2]{kw2018} written in terms of the discrete functions $\m_h^\ell, \v_h^\ell, \m_h^{\ell+1} \in \SS^1(\TT_h)^3$.
In particular, the corrector~\eqref{eq:pc2:variational:corrector} of Algorithm~\ref{alg:pc2:variational} reformulates the $N$ equations in $\R^3$ of the corrector of~\cite[Algorithm~2]{kw2018} as an equivalent variational formulation for $\m_h^{\ell+1}$ in $\SS^1(\TT_h)^3$.
The predictor step coincides with step~(i) of Algorithm~\ref{alg:pc1:variational}, i.e., \eqref{eq:pc2:variational:predictor} coincides with~\eqref{eq:pc1:variational:predictor}.
As in Section~\ref{sec:pc1}, the parameter $\theta \in [0,1]$ modulates the `degree of implicitness' (in the predictor) of the scheme.
\begin{algorithm}[\texttt{PC2}, variational form] \label{alg:pc2:variational}
\textbf{Input:}
$\m_h^0 \in \Mh$. \\
\textbf{Loop:}
For all time-steps $\ell = 0, \dots, L-1$, iterate:
\begin{itemize}
\item[\rm(i)] Compute $\v_h^\ell \in \SS^1(\TT_h)^3$ such that, for all $\w_h \in \SS^1(\TT_h)^3$, it holds that
\begin{equation} \label{eq:pc2:variational:predictor}
\begin{split}
(1 + \alpha^2) \inner[h]{\v_h^\ell}{\w_h}
&= -\lex^2 \inner[h]{\m_h^\ell \times \Lapl_h (\m_h^\ell + \theta k \v_h^\ell)}{\w_h} \\
&\quad -\alpha\lex^2 \inner[h]{\m_h^\ell \times (\m_h^\ell \times \Lapl_h (\m_h^\ell + \theta k \v_h^\ell))}{\w_h}\,.
\end{split}
\end{equation}
\item[\rm(ii)] Compute $\m_h^{\ell+1} \in \Mh$ such that, for all $\w_h \in \SS^1(\TT_h)^3$, it holds that
\begin{align} \label{eq:pc2:variational:corrector}
(1 + \alpha^2) \inner[h]{d_t\m_h^{\ell+1}}{\w_h}
&= -\lex^2 \inner[h]{\m_h^{\ell+1/2} \times \Lapl_h (\m_h^\ell + (k/2) \v_h^\ell)}{\w_h} \\
\notag&\quad -\alpha\lex^2 \inner[h]{\m_h^{\ell+1/2} \times [(\m_h^\ell + (k/2) \v_h^\ell) \times \Lapl_h (\m_h^\ell + (k/2) \v_h^\ell)]}{\w_h}\,.
\end{align}
\end{itemize}
\textbf{Output:}
Sequence of discrete functions $\left\{(\v_h^\ell,\m_h^{\ell+1})\right\}_{\ell= 0}^{L-1}$.
\end{algorithm}
The corrector step of Algorithm~\ref{alg:pc1:variational}, which combines a linear first-order time-stepping with the nodal projection, is replaced by the linear system~\eqref{eq:pc2:variational:corrector}.
The 2D numerical results of \cite[Figure~3]{kw2018} indicate that the method is of second-order in time.
In Section~\ref{sec:convOrder}, we confirm this observation for a numerical example in 3D.

\subsection{Unconditional well-posedness, exact solver}
\label{sec:pc2:well_posedness}
In Lemma~\ref{lemma:pc2:Mh_Kh}, we first collect two basic properties of Algorithm~\ref{alg:pc2:variational}, which, for $\alpha > 0$, turn out to be sufficient to prove unconditional well-posedness of the algorithm in Theorem~\ref{thm:pc2:well_posedness}.
\begin{lemma}\label{lemma:pc2:Mh_Kh}
Let $\m_h^\ell \in \Mh$.
Suppose that the solutions $\v_h^\ell \in \SS^1(\TT_h)^3$ and $\m_h^{\ell+1} \in \SS^1(\TT_h)^3$ to~\eqref{eq:pc2:variational:predictor} and \eqref{eq:pc2:variational:corrector} exist, respectively.
Then, $\v_h^\ell \in \Kh[\m_h^\ell]$, and $\m_h^{\ell+1} \in \Mh$.
\end{lemma}
\begin{proof}
The claim $\v_h^{\ell} \in \Kh[\m_h^\ell]$ follows as in the proof of Lemma~\ref{lemma:pc1:Mh_Kh}.
We show that $\m_h^{\ell} \in \Mh$ implies $\m_h^{\ell+1} \in \Mh$ due to the corrector system~\eqref{eq:pc2:variational:corrector}:
For arbitrary $\z \in \NN_h$, we choose $\w_h := \m_h^{\ell+1/2}(\z)\phi_{\z} \in \SS^1(\TT_h)^3$ in~\eqref{eq:pc2:variational:corrector} to see
\begin{align*}
\frac{(1 + \alpha^2) \beta_{\z}}{2k} \Big(|\m_h^{\ell+1}(\z)|^2 - |\m_h^{\ell}(\z)|^2 \Big)
\stackrel{\eqref{eq:mass-lumping}}{=} (1 + \alpha^2) \inner[h]{d_t\m_h^{\ell+1}}{\m_h^{\ell+1/2}(\z)\phi_{\z}}
\stackrel{\eqref{eq:pc2:variational:corrector}, \eqref{eq:cross:b}}{=} 0 \,.
\end{align*}
This shows that $|\m_h^{\ell+1}(\z)| = |\m_h^{\ell}(\z)|$ for all $\z \in \NN_h$.
Hence, $\m_h^\ell \in \Mh$ implies that $\m_h^{\ell+1} \in \Mh$.
The assumption $\m_h^0 \in \Mh$ concludes the proof.
\end{proof}
We show unconditional well-posedness of the corrector~\eqref{eq:pc2:variational:corrector}, while with Lemma~\ref{lemma:pc2:Mh_Kh} unconditional well-posedness of the predictor is inferred from our analysis in Section~\ref{sec:pc1:well_posedness}.
\begin{theorem}\label{thm:pc2:well_posedness}
Let $\alpha > 0$.
Then, Algorithm~\ref{alg:pc2:variational} is unconditionally well-posed for any input $\m_h^0 \in \Mh$, i.e., for all $\ell = 0, \dots, L-1$, the predictor~\eqref{eq:pc2:variational:predictor} admits a unique solution $\v_h^\ell \in \SS^1(\TT_h)^3$, and the corrector~\eqref{eq:pc2:variational:corrector} admits a unique solution $\m_h^{\ell+1} \in \Mh$.
\end{theorem}
\begin{proof}
By Lemma~\ref{lemma:pc2:Mh_Kh} it holds that $\m_h^\ell \in \Mh$ and $\v_h^\ell \in \Kh[\m_h^\ell]$ for all $\ell = 0, \dots, L - 1$.
Hence, as for the predictor of Algorithm~\ref{alg:pc1:variational}, the predictor system~\eqref{eq:pc2:variational:predictor} is equivalent to a coercive system in the discrete tangent space $\Kh[\m_h^\ell]$ with unique solution $\v_h^\ell \in \Kh[\m_h^\ell]$; see (the proof of) Theorem~\ref{thm:pc1:well_posedness}.
It remains to show well-posedness of the corrector~\eqref{eq:pc2:variational:corrector}:
We rewrite the problem in terms of the unknown $\eeta_h^\ell := \m_h^{\ell+1/2}$, which, by construction, satisfies that $\m_h^{\ell+1} = 2\eeta_h^\ell - \m_h^{\ell}$ and $d_t \m_h^{\ell+1} = 2(\eeta_h^\ell - \m_h^{\ell})/k$.
The corrector system~\eqref{eq:pc2:variational:corrector} then reads:
Find $\eeta_h^\ell \in \SS^1(\TT_h)^3$ such that
\begin{equation*}
\begin{split}
a_{\operatorname{cor}}[\m_h^\ell, \v_h^\ell]( \eeta_h^\ell, \w_h )
= (1 + \alpha^2) \inner[h]{\m_h^{\ell}}{\w_h},
\end{split}
\end{equation*}
where the bilinear form $a_{\operatorname{cor}}[\m_h^\ell, \v_h^\ell] \colon \SS^1(\TT_h)^3 \times \SS^1(\TT_h)^3 \to \R$ is defined by
\begin{equation*}
\begin{split}
a_{\operatorname{cor}}[\m_h^\ell, \v_h^\ell]( \eeta_h^\ell, \w_h )
& := 
(1 + \alpha^2) \inner[h]{\eeta_h^\ell}{\w_h}
+ \frac{\lex^2 k}{2} \, \inner[h]{\eeta_h^\ell \times \Lapl_h(\m_h^\ell + (k/2) \v_h^\ell)}{\w_h} \\
& \quad + \frac{\alpha \lex^2 k}{2} \, \inner[h]{\eeta_h^\ell \times [ (\m_h^\ell + (k/2) \v_h^\ell) \times \Lapl_h(\m_h^\ell + (k/2) \v_h^\ell) ]}{\w_h}.
\end{split}
\end{equation*}
As the bilinear form satisfies the ellipticity property
\begin{equation*}
a_{\operatorname{cor}}[\m_h^\ell, \v_h^\ell]( \w_h, \w_h)
= 
(1 + \alpha^2) \norm[h]{\w_h}^2
\quad
\text{for all } \w_h \in \SS^1(\TT_h)^3,
\end{equation*}
the problem is well-posed by the Lax--Milgram theorem.
Hence, \eqref{eq:pc2:variational:corrector} provides a unique solution $\m_h^{\ell+1} \in \SS^1(\TT_h)^3$.
Lemma~\ref{lemma:pc2:Mh_Kh} guarantees $\m_h^{\ell+1} \in \Mh$ concluding the proof.
\end{proof}
\begin{remark}\label{re:pc2:nature}
{\rm(i)} Algorithm~\ref{alg:pc2:variational} is a predictor-corrector scheme:
Both systems, for the predictor~\eqref{eq:pc2:variational:predictor} and for the corrector~\eqref{eq:pc2:variational:corrector}, respectively, are linear systems representing discrete mass-lumped variational versions of the LL form~\eqref{eq:llg:ll} of LLG; see also Remark~\ref{re:pc1:nature}{\rm(i)}.
First, treating the effective field implicitly in time, an approximate time derivative $\v_h^\ell \in \Kh[\m_h^\ell]$, the predictor, is computed.
In the second step (the effective field of) the predicted midpoint $\m_h^\ell + (k/2)\v_h^\ell \in \SS^1(\TT_h)^3$ is used to compute a corrected update $d_t\m_h^{\ell+1} \in \SS^1(\TT_h)^3$, guaranteeing conservation of discrete unit-length $\m_h^{\ell+1} := \m_h^\ell + k d_t\m_h^{\ell+1} \in \Mh$.\\
{\rm(ii)} In the proof of Theorem~\ref{thm:pc2:well_posedness}, note that the assumption $\alpha > 0$ is only exploited to apply Theorem~\ref{thm:pc1:well_posedness}.
Hence, analogously to Theorem~\ref{thm:pc1:well_posedness} (Remark~\ref{re:pc1:nature}{\rm(iv)}), also Theorem~\ref{thm:pc2:well_posedness} can be extended to the limit case $\alpha = 0$; see Theorem~\ref{thm:pc2:inexact:well_posedness} below.
\end{remark}

\subsection{Unconditional well-posedness, inexact solver}
\label{sec:pc2:inexact}
Considering the effect of numerical approximations, we extend the theoretical well-posedness result from the previous section to the practical case.

Well-posedness of the predictor step~{\rm(i)} of Algorithm~\ref{alg:pc2:variational} is guaranteed by Theorem~\ref{thm:pc1:well_posedness}:
There, under the crucial condition $\m_h^\ell \in \Mh$, computing $\v_h^\ell$ in the predictor step is shown to be equivalent to solving the system~\eqref{eq:pc1:tpsh} in the discrete tangent space $\Kh[\m_h^\ell]$, which is always well-posed for $\alpha > 0$.
While $\m_h^\ell \in \Mh$ is explicitly enforced in step~{\rm(ii)} of Algorithm~\ref{alg:pc1:variational}, in Algorithm~\ref{alg:pc2:variational} it follows only implicitly from the inherent length preservation guaranteed by the variational formulation~\eqref{eq:pc2:variational:corrector} solved in step~{\rm(ii)} together with $\m_h^{\ell-1} \in \Mh$ in the previous time-step; see the proof of Lemma~\ref{lemma:pc2:Mh_Kh}.
In practice however, linear systems are solved by inexact (iterative) numerical solvers, i.e., the coefficient vector of the unknown $\m_h^{\ell+1}$ solves the linear system of equations corresponding to~\eqref{eq:pc2:variational:corrector} only up to some accuracy $\varepsilon > 0$, commonly in the $\ell^2(\R^{3N})$-norm.
Consequently, for any $\z \in \NN_h$ there only holds $|\m_h^{\ell+1}(\z)| \approx |\m_h^{\ell}(\z)|$ with a small error depending on the discretization parameters $\varepsilon$ and $h$.
Moreover, the deviation from nodewise unit-length accumulates over the time-steps $\ell = 0, \dots, L-1$.
Consequently --- if recoverable at all --- one expects to require CFL-type couplings of the discretization parameters $k, h, \eps$ to rigorously argue (approximate) equivalence of the linear system in step~{\rm(i)} of Algorithm~\ref{alg:pc2:variational} and the well-posed system~\eqref{eq:pc1:tpsh} in the proof of Theorem~\ref{thm:pc1:well_posedness}.

To avoid these analytical difficulties, we take a different analytical approach:
The new analysis uses a space decomposition technique reformulating~\eqref{eq:pc2:variational:predictor} as an equivalent saddle-point problem, which subsequently is proved to be unconditionally well-posed and hence always provides a unique solution.
In particular, this does not require $\m_h^\ell \in \Mh$, but allows for arbitrary $\m_h^\ell \in \SS^1(\TT_h)^3 \supsetneqq \Mh$.
Additionally, the analysis applies to all $\alpha \ge 0$, extending well-posedness of Algorithm~\ref{alg:pc2:variational} to the Schr\"odinger map equation ($\alpha = 0$).
\begin{theorem}\label{thm:pc2:inexact:well_posedness}
Let $\alpha \ge 0$.
Then, Algorithm~\ref{alg:pc2:variational} is unconditionally well-posed for any input $\m_h^0 \in \SS^1(\TT_h)^3$, i.e., for all $\ell = 0, \dots, L-1$ and any $\m_h^\ell \in \SS^1(\TT_h)^3$, the predictor~\eqref{eq:pc2:variational:predictor} admits a unique solution $\v_h^\ell \in \SS^1(\TT_h)^3$, and the corrector~\eqref{eq:pc2:variational:corrector} admits a unique solution $\m_h^{\ell+1} \in \SS^1(\TT_h)^3$.
\end{theorem}
\begin{proof}
For arbitrary $\m_h^\ell \in \SS^1(\TT_h)^3$ well-posedness of the corrector~\eqref{eq:pc2:variational:corrector} is guaranteed by the proof of Theorem~\ref{thm:pc2:well_posedness}, as it does not require $\m_h^\ell \in \Mh$.
Using a space decomposition technique, we show unconditional well-posedness of the predictor system~\eqref{eq:pc2:variational:predictor} for any $\m_h^\ell \in \SS^1(\TT_h)^3$ --- in particular for $\m_h^\ell \in \SS^1(\TT_h)^3$ not necessarily belonging to $\Mh$ --- in five steps:
\begin{itemize}
\item \textbf{Step~0:} Some notation.
\end{itemize}
Throughout, for an operator $A \colon X \to Y$ between two Hilbert spaces, we write $\Range(A) \subseteq Y$ for its range, and $\Kernel(A) \subseteq X$ for its kernel.
We consider the (negative) discrete Laplace operator~\eqref{eq:discrete_laplacian} restricted to $\SS^1(\TT_h)^3 \subset \H^1(\Omega)$, which will be denoted by the same symbol $-\Lapl_h \colon \SS^1(\TT_h)^3 \to \SS^1(\TT_h)^3$.
Further, we identify a $3$-vector with the corresponding constant vector-valued grid function, i.e., $\R^3 \subset \big(\SS^1(\TT_h)^3, \inner[h]{\cdot}{\cdot}\big)$.
For $S \subset \SS^1(\TT_h)^3$ a subspace we denote by $\II_S$ the identity on $S$.
\begin{itemize}
\item \textbf{Step~1:} Orthogonal decomposition $\SS^1(\TT_h)^3 = \Range(\PPast) \oplus \Kernel(\PPast)$.
\end{itemize}
Define the operator $\PPast \colon \SS^1(\TT_h)^3 \to \SS^1(\TT_h)^3$ for all $\w_h \in \SS^1(\TT_h)^3$ via
\begin{align*}
(\PPast\w_h)_j = (\w_h)_j - \meas(\Omega)^{-1}\inner[h]{\w_h}{\e_j} \in \SS^1(\TT_h) \qquad\text{for all}\quad j = 1, 2, 3\,.
\end{align*}
Clearly, $\PPast$ is the $\inner[h]{\cdot}{\cdot}$-orthogonal projector onto
\begin{align*}
\Range(\PPast) = \SS_\ast^1(\TT_h)^3 := \{\w_h \in \SS^1(\TT_h)^3 \colon \inner[h]{\w_h}{\e_j} = 0 \text{ for all } j = 1, 2, 3\}\,,
\end{align*} 
the subset of $\SS^1(\TT_h)^3$ consisting of the vector-valued grid functions which have zero mean in each component.
Due to self-adjointness, $\PPast$ provides the orthogonal decomposition
\begin{align*}
\SS^1(\TT_h)^3 = \Range(\PPast) \oplus \Kernel(\PPast) = \SS_\ast^1(\TT_h)^3 \oplus \R^3 \,.
\end{align*}
With respect to this decomposition, we rewrite the unknown $\v_h^\ell \in \SS^1(\TT_h)^3$ as the orthogonal sum 
\begin{align}\label{eq:v:decomp:PPast}
\v_h^\ell = \PPast\v_h^\ell \oplus (\II_{\SS^1(\TT_h)^3} - \PPast)\v_h^\ell =: \v_\ast \oplus \overline{\v}\,,
\end{align}
with unique $\v_\ast \in \Range(\PPast) = \SS_\ast^1(\TT_h)^3$ and $\overline{\v} \in \Kernel(\PPast) = \R^3$.
Note, that $\overline{\v} \in \R^3$ is the vector-valued mean of $\v_h^\ell$, i.e., $\inner[h]{\overline{\v}}{\e_j} = \inner[h]{\v_h^\ell}{\e_j}$ for all components $j=1, 2, 3$.
\begin{itemize}
\item \textbf{Step~2:} Reduced operator $-\RLapl_h \colon \SS_\ast^1(\TT_h)^3 \to \SS_\ast^1(\TT_h)^3$.
\end{itemize}
The discrete Laplacian $-\Lapl_h \colon \SS^1(\TT_h)^3 \to \SS^1(\TT_h)^3$ is linear, self-adjoint and by definition~\eqref{eq:discrete_laplacian} has the kernel $\NN(-\Lapl_h) = \R^3 \subset \SS^1(\TT_h)^3$.
Hence, there holds the orthogonal decomposition
\begin{align}\label{eq:decomposition:laplace}
\SS^1(\TT_h)^3 
= \Range(-\Lapl_h) \oplus \Kernel(-\Lapl_h)
= \Kernel(-\Lapl_h)^\bot \oplus \Kernel(-\Lapl_h) 
= \SS_\ast^1(\TT_h)^3 \oplus \R^3 \,.
\end{align}
Consequently, the reduced operator $-\Lapl_h|_{\SS_\ast^1(\TT_h)^3} =: -\RLapl_h \colon \SS_\ast^1(\TT_h)^3 \to \SS_\ast^1(\TT_h)^3$ is linear, self-adjoint, and bijective.
Moreover, it provides a well-defined inverse denoted by $(-\RLapl_h)^{-1} \colon \SS_\ast^1(\TT_h)^3 \to \SS_\ast^1(\TT_h)^3$ with the same attributes.
We point out the identities
\begin{align}\label{eq:RLapl:identities}
(-\RLapl_h)^{-1} \circ (-\Lapl_h) = \PPast \quad\text{ and }\quad
-\Lapl_h \circ (-\RLapl_h)^{-1} = \PPast|_{\SS_\ast^1(\TT_h)^3} = \II_{\SS_\ast^1(\TT_h)^3}\,,
\end{align}
which follow from the orthogonal decomposition~\eqref{eq:decomposition:laplace}.
\begin{itemize}
\item \textbf{Step~3:} Equivalent saddle point formulation.
\end{itemize}
With the unknowns $\q := -\Lapl_h \v_\ast \in \SS_\ast^1(\TT_h)^3$ and $\llambda := \overline{\v} \in \R^3$ from~\eqref{eq:v:decomp:PPast}, we induce the representation $\v_h^\ell = (-\RLapl_h)^{-1}\q \oplus \llambda$.
Plugging this identity into~\eqref{eq:pc2:variational:predictor}, we rewrite the predictor as equivalent saddle point problem:
Find $(\q, \llambda) \in \SS^1(\TT_h)^3 \times \R^3$, such that for all $(\w, \mmu) \in \SS^1(\TT_h)^3 \times \R^3$ it holds that
\begin{subequations}
\begin{align}
\label{eq:saddle_point:llg} a_{\operatorname{sp}}[\m_h^\ell](\q, \w) + b_{\operatorname{sp}}(\w, \llambda) &= F_{\operatorname{sp}}[\m_h^\ell](\w)\,, \\
\label{eq:saddle_point:zeromean} b_{\operatorname{sp}}(\q, \mmu) &= 0 \,,
\end{align}
\end{subequations}
with the (bi-)linear forms $a_{\operatorname{sp}}[\m_h^\ell] \colon \SS^1(\TT_h)^3 \times \SS^1(\TT_h)^3 \to \R$, $b \colon \SS^1(\TT_h)^3 \times \R^3 \to \R$, and $F_{\operatorname{sp}}[\m_h^\ell] \colon \SS^1(\TT_h)^3 \to \R$ given by
\begin{align*}
a_{\operatorname{sp}}[\m_h^\ell](\q, \w) 
&:= (1 + \alpha^2) \inner[h]{(-\RLapl_h)^{-1}\PPast\q}{\w} \\
&\quad -\lex^2 \theta k \inner[h]{\m_h^\ell \times \q}{\w} 
- \alpha\lex^2\theta k \inner[h]{\m_h^\ell \times (\m_h^\ell \times \q)}{\w}\,, \\
b_{\operatorname{sp}}(\w, \llambda) &:= (1 + \alpha^2) \inner[h]{\llambda}{\w} \,, \\
F_{\operatorname{sp}}[\m_h^\ell](\w) 
&:= -\lex^2 \inner[h]{\m_h^\ell \times \Lapl_h\m_h^\ell}{\w} 
- \alpha\lex^2 \inner[h]{\m_h^\ell \times (\m_h^\ell \times \Lapl_h\m_h^\ell)}{\w}\,.
\end{align*}
The equivalence of~\eqref{eq:saddle_point:llg}--\eqref{eq:saddle_point:zeromean} to~\eqref{eq:pc2:variational:predictor} follows from $\llambda \in \Kernel(-\Lapl_h)$ and~\eqref{eq:RLapl:identities}.
We use the operator $(-\RLapl_h)^{-1} \circ \PPast$ rather than $(-\RLapl_h)^{-1}$, so that the bilinear form $a_{\operatorname{sp}}[\m_h^\ell]$ is well-defined on $\SS^1(\TT_h)^3 \supsetneqq \SS_\ast^1(\TT_h)^3$.
The second equation~\eqref{eq:saddle_point:zeromean} ensures $\q \in \SS_\ast^1(\TT_h)^3$, which is not enforced explicitly.
\begin{itemize}
\item \textbf{Step~4:} The bilinear form $a_{\operatorname{sp}}[\m_h^\ell]$ is coercive on the kernel of $b_{\operatorname{sp}}$.
\end{itemize}
We aim to apply the Brezzi theory for saddle point problems; see, e.g.,~\cite[Section~4.2]{bbf2013}.
Hence, we require coercivity of the bilinear form $a_{\operatorname{sp}}[\m_h^\ell] \colon \SS^1(\TT_h)^3 \times \SS^1(\TT_h)^3 \to \R$ on
\begin{align*}
\bigcap_{\llambda \in \R^3} \Kernel\big(b_{\operatorname{sp}}(\cdot, \llambda)\big)
&= \bigcap_{\llambda \in \R^3} \{\w \in \SS^1(\TT_h)^3 \colon \inner[h]{\llambda}{\w} = 0\} \\
&= \bigcap_{j = 1, 2, 3} \{\w \in \SS^1(\TT_h)^3 \colon \inner[h]{\e_j}{\w} = 0\}
= \SS_\ast^1(\TT_h)^3 \,.
\end{align*}
For any $\q \in \SS_\ast^1(\TT_h)^3$, we compute
\begin{align*}
a_{\operatorname{sp}}[\m_h^\ell](\q, \q) 
&\stackrel{\eqref{eq:cross:b}}{=} (1 + \alpha^2) \inner[h]{(-\RLapl_h)^{-1}\PPast\q}{\q} - \alpha\lex^2\theta k \inner[h]{\m_h^\ell \times (\m_h^\ell \times \q)}{\q}\,, \\
&\stackrel{\eqref{eq:RLapl:identities}, \eqref{eq:cross:d}}{=} (1 + \alpha^2) \inner[h]{(-\RLapl_h)^{-1}\q}{-\Lapl_h(-\RLapl_h)^{-1}\q} + \alpha\lex^2\theta k\norm[h]{\m_h^\ell\times\q}^2 \\
&\stackrel{\eqref{eq:discrete_laplacian}}{=} (1 + \alpha^2) \bignorm[\L^2(\Omega)]{\Grad(-\RLapl_h)^{-1}\q}^2 + \alpha\lex^2\theta k\norm[h]{\m_h^\ell\times\q}^2 \\
&\gtrsim h^2\norm[h]{\q}^2 + \alpha\lex^2\theta k\norm[h]{\m_h^\ell\times\q}^2
\ge h^2\norm[h]{\q}^2 \,,
\end{align*}
where the second to last estimate is an inverse estimate on $\SS_\ast^1(\TT_h)^3$ derived from the classical inverse estimate on $\SS^1(\TT_h)^3$ via
\begin{align*}
\norm[h]{\q}^2
&= \inner[h]{\q}{\q}
\stackrel{\eqref{eq:RLapl:identities}}{=} \inner[h]{\q}{-\Lapl_h(-\RLapl_h)^{-1}\q}
\stackrel{\eqref{eq:discrete_laplacian}}{=} \inner[\L^2(\Omega)]{\Grad\q}{\Grad(-\RLapl_h)^{-1}\q} \\
&\le \norm[\L^2(\Omega)]{\Grad\q}\bignorm[\L^2(\Omega)]{\Grad(-\RLapl_h)^{-1}\q}
\lesssim h^{-1}\norm[h]{\q}\bignorm[\L^2(\Omega)]{\Grad(-\RLapl_h)^{-1}\q}\,.
\end{align*}
Hence, $a_{\operatorname{sp}}[\m_h^\ell]$ is coercive on $\bigcap_{\llambda \in \R^3}\NN(b_{\operatorname{sp}}(\cdot, \llambda)) = \SS_\ast^1(\TT_h)^3$ with ellipticity constant proportional to $h^2 > 0$.
\begin{itemize}
\item \textbf{Step~5:} Unique solvability and reconstruction of $\v_h^\ell$.
\end{itemize}
Clearly, $b_{\operatorname{sp}} \colon \SS^1(\TT_h)^3 \times \R^3 \to \R$ satisfies the inf-sup condition with constant $(1 + \alpha^2) > 0$.
Now unique solvability of the saddle point formulation~\eqref{eq:saddle_point:llg}--\eqref{eq:saddle_point:zeromean} follows from the Brezzi theorem~\cite[Theorem~4.2.1]{bbf2013}.
Ultimately, with $(\q, \llambda) \in \SS_\ast^1(\TT_h)^3 \times \R^3$ denoting the unique solution of~\eqref{eq:saddle_point:llg}--\eqref{eq:saddle_point:zeromean}, the original unknown solution to~\eqref{eq:pc2:variational:predictor} is reconstructed via $\v_h^\ell = (-\RLapl_h)^{-1}\q \oplus \llambda \in \SS_\ast^1(\TT_h)^3 \oplus \R^3 = \SS^1(\TT_h)^3$ and is therefore also unique.
\end{proof}
\begin{remark}\label{re:proof:inexact}
{\rm(i)} In the third step of the proof of Theorem~\ref{thm:pc2:inexact:well_posedness}, we introduced the unknown $\q := -\Lapl_h\v_\ast \in \SS_\ast^1(\TT_h)^3$.
This idea is inspired by~\cite[Section~2.3]{xgwzc2020}, where the authors subsequently use the Browder--Minty lemma for monotone operators to prove well-posedness of their proposed finite difference LLG integrator based on the second-order backward differentiation formula. \\
{\rm(ii)} In Step~4 of the proof of Theorem~\ref{thm:pc2:inexact:well_posedness}, as the new unknown $\q = -\Lapl_h\v_h^\ell$ comprises second-order derivatives of the original unknown, it is not surprising that the ellipticity constant for the bilinear form $a_{\operatorname{sp}}[\m_h^\ell]$ scales proportionally to $h^2 > 0$.\\
{\rm(iii)} Since $\Mh \subset \SS^1(\TT_h)^3$ and the predictors of Algorithm~\ref{alg:pc1:variational} and Algorithm~\ref{alg:pc2:variational} coincide, the proof of Theorem~\ref{thm:pc2:inexact:well_posedness} is not only an alternative proof to Theorem~\ref{thm:pc2:well_posedness}, but also to Theorem~\ref{thm:pc1:well_posedness}, which additionally extends both theorems to the critical value $\alpha = 0$.\\
{\rm(iv)} Consequently, Algorithm~\ref{alg:pc1:variational} is not only a mass-lumped version of the tangent plane scheme~\cite{alouges2008a}, but additionally it is well-posed for the Schr\"odinger map equation ($\alpha = 0$).\\
{\rm(v)} Even though the predictor of Algorithm~\ref{alg:pc1:variational} written in the form~\eqref{eq:pc1:tpsh} coincides with the predictor of the tangent plane scheme up to the used integration rule, well-posedness of the tangent plane scheme for the limit case $\alpha = 0$ remains open.
Indeed, the proof of Theorem~\ref{thm:pc2:inexact:well_posedness} relies heavily on mass-lumped integration, and we did not succeed to transfer the proof to exact integration used in the original tangent plane scheme.
\end{remark}

\subsection{Including the lower-order contributions} \label{sec:pc2:imex}
We consider the case when the effective field comprises linear lower-order energy contributions $\ppi(\m)$ such as, in particular, the nonlocal stray field $\hstray$, i.e., $\heff(\m) = \lex^2 \, \Lapl \m + \ppi(\m) + \f$.
Then the predictor step of the original second-order integrator proposed in~\cite[Algorithm~2]{kw2018} is identical to~\eqref{eq:pc1:implicit_pi}, i.e., lower-order terms are treated implicitly in time.
Due to the nonlocality of the stray field this is unattractive in practice as described in Section~\ref{sec:pc1:imex}. 
Hence, analogously to Section~\ref{sec:pc1:imex}, we aim to treat the lower-order terms $\ppi(\m)$ explicitly in time.
However, to avoid spoiling the scheme's potential second-order accuracy in time, which was observed experimentally in~\cite{kw2018}, the modification is slightly more involved:

In Section~\ref{sec:pc1:imex} an error of order $\mathcal{O}(k)$ is introduced to the system~\eqref{eq:pc1:implicit_pi} by approximating $\ppi_h(\m_h^\ell + \theta k\v_h^\ell) \approx \ppi_h(\m_h^\ell)$.
Since Algorithm~\ref{alg:pc1:variational} is a first-order scheme, this modification did not deteriorate the order of convergence of the algorithm.

To preserve the potential second-order of Algorithm~\ref{alg:pc2:variational}, we use a higher-order approximation to $\ppi(\m_h^\ell + \theta k\v_h^\ell)$:
Recall, that $\ppi$ is a linear operator and that $\v_h^\ell$ is an approximation of $\partial_t\m(t_\ell)$.
Motivated by the Taylor expansion $\m(t_\ell) = \m(t_{\ell-1}) + k\partial_t\m(t_\ell) + \mathcal{O}(k^2)$, and hence $\m(t_\ell) + \theta k\partial_t\m(t_\ell) = (1 + \theta)\m(t_\ell) - \theta\m(t_{\ell-1}) + \mathcal{O}(k^2)$, we introduce a second-order error $\mathcal{O}(k^2)$ to the system~\eqref{eq:pc1:implicit_pi} via the approximation
\begin{equation*}
\ppi_h(\m_h^\ell + \theta k\v_h^\ell) \approx (1 + \theta)\ppi_h(\m_h^\ell) - \theta\ppi_h(\m_h^{\ell-1})\,.
\end{equation*}
Only the leading-order exchange contribution is treated implicitly in time, while the lower-order contributions are treated explicitly.
Due to the higher-order approximation of $\ppi_h(\v_h^\ell)$, this does not spoil the observed second-order of the scheme and it is computationally much more attractive.
To sum up, we consider the following algorithm.
\begin{algorithm}[\texttt{PC2+IMEX}] \label{alg:pc2:imex}
\textbf{Input:}
$\m_h^0 \in \Mh$. \\
\textbf{Preprocessing:} Compute $\m_h^1 \in \Mh$, e.g., by Algorithm~\ref{alg:pc2:variational}.\\
\textbf{Loop:}
For all time-steps $\ell = 1, \dots, L-1$, iterate:
\begin{itemize}
\item[\rm(i)] Compute $\P_h((1+\theta)\ppi_h(\m_h^\ell) - \theta\ppi_h(\m_h^{\ell-1}) + \f^{\ell + \theta}) \in \SS^1(\TT_h)^3$.
\item[\rm(ii)] Compute $\v_h^\ell \in \SS^1(\TT_h)^3$ such that, for all $\w_h \in \SS^1(\TT_h)^3$, it holds that
\begin{align} \label{eq:pc2:imex:predictor}
&(1 + \alpha^2) \inner[h]{\v_h^\ell}{\w_h} \\
\notag&\;= -\inner[h]{\m_h^\ell \times [\lex^2\Lapl_h (\m_h^\ell + \theta k \v_h^\ell) + \P_h((1+\theta)\ppi_h(\m_h^\ell) - \theta\ppi_h(\m_h^{\ell-1}) + \f^{\ell + \theta})]}{\w_h} \\
\notag&\;\,-\alpha\inner[h]{\m_h^\ell \times (\m_h^\ell \times [\lex^2\Lapl_h (\m_h^\ell + \theta k \v_h^\ell) + \P_h((1+\theta)\ppi_h(\m_h^\ell) - \theta\ppi_h(\m_h^{\ell-1}) + \f^{\ell + \theta})])}{\w_h}\,.
\end{align}
\item[\rm(iii)] Compute $\m_h^{\ell+1} \in \Mh$ such that, for all $\w_h \in \SS^1(\TT_h)^3$, it holds that
\begin{equation*}
\begin{split}
& (1 + \alpha^2) \inner[h]{d_t\m_h^{\ell+1}}{\w_h} \\
&\;\;= - \inner[h]{\m_h^{\ell+1/2} \times [\lex^2\Lapl_h(\m_h^\ell + (k/2) \v_h^\ell) + \P_h(\ppi_h(\m_h^\ell + (k/2)\v_h^\ell) + \f^{\ell + 1/2})]}{\w_h} \\
&\quad\;\; - \alpha \inner[h]{\m_h^{\ell+1/2} \times \big((\m_h^\ell + (k/2) \v_h^\ell) \\
&\quad\;\;\;\; \times [\lex^2\Lapl_h(\m_h^\ell + (k/2) \v_h^\ell) + \P_h(\ppi_h(\m_h^\ell + (k/2)\v_h^\ell) + \f^{\ell + 1/2})]\big)}{\w_h}.
\end{split}
\end{equation*}
\end{itemize}
\textbf{Output:}
Sequence of discrete functions $\left\{(\v_h^\ell,\m_h^{\ell+1})\right\}_{\ell= 0}^{L-1}$.
\end{algorithm}
\begin{remark}
{\rm(i)} In the preprocessing step of Algorithm~\ref{alg:pc2:imex} also other integrators may be used to compute $\m_h^1 \in \Mh$.
As long as the approximation $\m_h^1$ is second-order accurate, the potential second-order accuracy of Algorithm~\ref{alg:pc2:variational} is preserved by Algorithm~\ref{alg:pc2:imex}.
(Note that first-order accurate integrators usually only introduce a quadratic error per time-step.)\\
{\rm(ii)} Algorithm~\ref{alg:pc2:imex} is also well-posed in practice, when effects of inexact (iterative) solvers are accounted for, i.e.,~\eqref{eq:pc2:imex:predictor} is unconditionally well-posed for arbitrary $\m_h^\ell \in \SS^1(\TT_h)^3 \supsetneqq \Mh$.
As lower-order terms are treated explicitly in time, proving well-posedness follows the lines of the proof of Theorem~\ref{thm:pc2:inexact:well_posedness} with adjusted linear form $F_{\operatorname{sp}}[\m_h^\ell] \rightsquigarrow F_{\operatorname{imex}}[\m_h^\ell, \m_h^{\ell-1}]$.
\end{remark}

\section{Numerical experiments} 
\label{sec:numerics}
This section provides some numerical experiments for Algorithm~\ref{alg:pc1:variational} and Algorithm~\ref{alg:pc2:variational} from~\cite{kw2018}, as well as their respective IMEX versions proposed in this work, namely Algorithm~\ref{alg:pc1:imex} and Algorithm~\ref{alg:pc2:imex}, respectively.
In Section~\ref{sec:mumag4} we verify the correctness of the proposed integrators (\texttt{PC1+IMEX} and \texttt{PC2+IMEX}) on the benchmark problem $\mu$MAG~\#4 from \cite{MUMAG}.
In Section~\ref{sec:convOrder} the experimental rates of Algorithm~\ref{alg:pc1:variational} (\texttt{PC1}) and Algorithm~\ref{alg:pc2:variational} (\texttt{PC2}) reported in~\cite{kw2018} are confirmed.
Moreover, the experiment shows that lower-order terms can appropriately be treated explicitly in time by Algorithm~\ref{alg:pc1:imex} (\texttt{PC1+IMEX}) and Algorithm~\ref{alg:pc2:imex} (\texttt{PC2+IMEX}), respectively, without spoiling the rate of convergence.

All computations have been performed with our micromagnetic software module Commics~\cite{commics2020},
based on the open-source finite element library Netgen/NGSolve~\cite{ngsolve}.
In Commics, the stray field $\hstray$ is computed via the hybrid FEM-BEM approach from~\cite{fk1990}.
We note that meshes generated by Netgen in general do not satisfy the angle condition~\eqref{eq:angleCondition}.
All experiments were repeated on structured meshes satisfying the angle condition leading to the same results (not displayed).

\subsection{\texorpdfstring{$\mu$}{mu}MAG standard problem~\#4}
\label{sec:mumag4}
We verify the practical applicability of the proposed integrators \texttt{PC1+IMEX} and \texttt{PC2+IMEX} (we choose $\theta=1/2$) by computing a physically relevant example.
To this end, we consider $\mu$MAG standard problem \#4~\cite{MUMAG}, which simulates the switching of the magnetization in a thin permalloy layer.

The objective is the simulation of the magnetization dynamics in a thin permalloy film of dimensions $\SI{500}{\nm}\times\SI{125}{\nm}\times\SI{3}{\nm}$ under the influence of a constant applied external field.
The involved physical constants and material parameters are the gyromagnetic ratio $\gamma_0 =$ \SI{2.211e5}{\meter\per\coulomb}, the permeability of vacuum $\mu_0 = 4 \pi \, \cdot \!\!$ \SI{e-7}{\newton/\ampere\squared}, the saturation magnetization $M_{\mathrm{s}} =$ \SI{8.0e5}{\ampere\per\meter}, the exchange stiffness constant $A =$ \SI{1.3e-11}{\joule\per\meter}, and the Gilbert damping constant $\alpha = 0.02$.
Starting from a so-called equilibrium S-state~\cite{MUMAG}, the experiment consists in applying
the constant applied field $\mu_0 \boldsymbol{H}_{\mathrm{ext}} = (-24.6  , 4.3 , 0 )$ \si{\milli\tesla}
for \SI{3}{\nano\second}.

For the rescaled form~\eqref{eq:llg:ibvp} of LLG, the above physical quantities lead to the parameters
$\lex = \sqrt{2A/(\mu_0 M_\mathrm{s}^2)}$,
$T = \num{3e-9} \gamma_0 M_{\mathrm{s}}$,
and $\f= \boldsymbol{H}_{\mathrm{ext}} / M_{\mathrm{s}}$,
while $\ppi(\m)$ includes only the stray field $\hstray$.
For the space discretization, we consider a tetrahedral partition of the thin film generated by Netgen~\cite{ngsolve} into cells of prescribed mesh size $\SI{3}{\nm}$.
This corresponds to $\num{48796}$ elements and $\num{16683}$ vertices.
For the time discretization, we consider a constant physical time-step size of $\Delta t = \SI{0.1}{\ps}$,
which is connected to the rescaled time-step size $k$ via the relation $k = \gamma_0 M_{\mathrm{s}} \Delta t$.

For comparison, the desired output of this benchmark problem is the evolution of the $x$-, $y$- and $z$-component of the spatially averaged magnetization.
Figure~\ref{fig:mumag4_avg} shows, that our results match those computed by the finite difference code OOMMF~\cite{OOMMF} available on the $\mu$MAG homepage~\cite{MUMAG}.
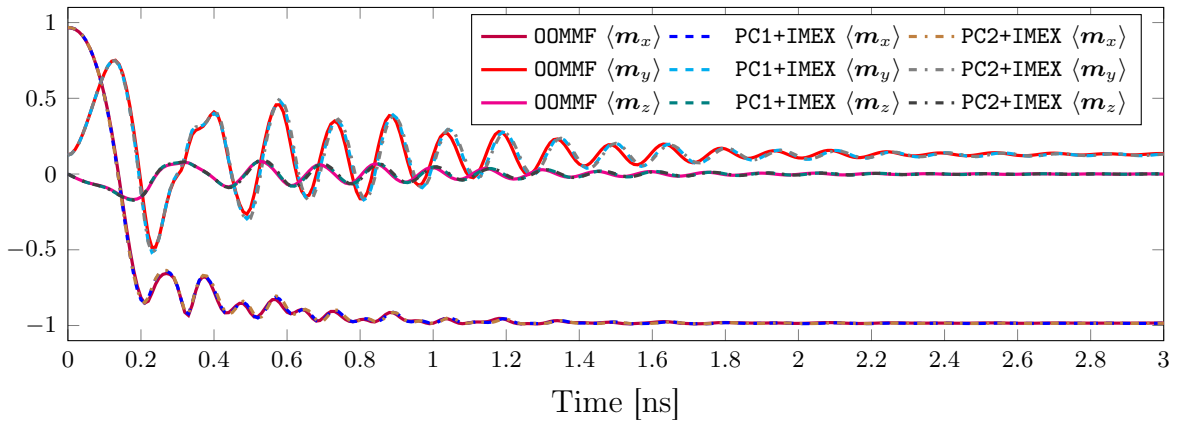
\begin{figure}[h]
\centering
\begin{tikzpicture}
\pgfplotstableread{plots/sp4-avg-oommf.dat}{\oommfData}
\pgfplotstableread{plots/sp4-avg-kw1ee.dat}{\kwFirstEE}
\pgfplotstableread{plots/sp4-avg-kw2ab.dat}{\kwSecondAB}
\begin{axis}[
width = 160mm,
height = 60mm,
xlabel={Time [\si{\nano\second}]},
xmin=0,
xmax=3,
ymin=-1.1,
ymax=1.1,
legend columns=3,
legend style={/tikz/column 3/.style={column sep=5pt}},
]
\addplot[purple, very thick] table[x=t, y=mx]{\oommfData};
\addplot[blue, very thick,dashed] table[x=t, y=mx]{\kwFirstEE};
\addplot[brown,very thick,dash pattern=on 3pt off 3pt on 1pt off 3pt] table[x=t, y=mx]{\kwSecondAB};
\addplot[red, very thick] table[x=t, y=my]{\oommfData};
\addplot[cyan,dashed, very thick] table[x=t, y=my]{\kwFirstEE};
\addplot[gray,very thick,dash pattern=on 3pt off 3pt on 1pt off 3pt] table[x=t, y=my]{\kwSecondAB};
\addplot[magenta, very thick] table[x=t, y=mz]{\oommfData};
\addplot[teal,dashed, very thick] table[x=t, y=mz]{\kwFirstEE};
\addplot[darkgray,very thick,dash pattern=on 3pt off 3pt on 1pt off 3pt] table[x=t, y=mz]{\kwSecondAB};
\legend{
\texttt{OOMMF} $\langle \m_x \rangle$,
\texttt{PC1+IMEX} $\langle \m_x \rangle$,
\texttt{PC2+IMEX} $\langle \m_x \rangle$,
\texttt{OOMMF} $\langle \m_y \rangle$,
\texttt{PC1+IMEX} $\langle \m_y \rangle$,
\texttt{PC2+IMEX} $\langle \m_y \rangle$,
\texttt{OOMMF} $\langle \m_z \rangle$,
\texttt{PC1+IMEX} $\langle \m_z \rangle$,
\texttt{PC2+IMEX} $\langle \m_z \rangle$,
}
\end{axis}
\end{tikzpicture}
\caption{$\mu$MAG standard problem~\#4 from Section~\ref{sec:mumag4}: Time evolution of the spatially averaged magnetization components computed with Algorithm~\ref{alg:pc1:imex} (\texttt{PC1+IMEX}) and Algorithm~\ref{alg:pc2:imex} (\texttt{PC2+IMEX}) compared to the results of {\tt OOMMF}.}
\label{fig:mumag4_avg}
\end{figure}

\subsection{Empirical convergence rates for LLG}
\label{sec:convOrder}%
We aim to illustrate the accuracy and the computational effort of the following four algorithms:
\begin{itemize}
\item {\tt PC1}: fully implicit first-order scheme proposed in~\cite{kw2018}
and recalled in Algorithm~\ref{alg:pc1:variational};
\item {\tt PC1+IMEX}: {\tt PC1} with explicit treatment of the lower-order terms as proposed in this work
and formulated in Algorithm~\ref{alg:pc1:imex};
\item {\tt PC2}: fully implicit second-order scheme proposed in~\cite{kw2018}
and recalled in Algorithm~\ref{alg:pc2:variational};
\item {\tt PC2+IMEX}: {\tt PC2} with explicit treatment of the lower-order terms as proposed in this work
and formulated in Algorithm~\ref{alg:pc2:imex};
\end{itemize}
For all integrators we choose $\theta=1/2$.
To obtain experimental convergence rates in time, we use the model problem proposed in~\cite{prs2017}:
We consider the initial boundary value problem~\eqref{eq:llg:ibvp}
with $\Omega = (0,1)^3$, $\m^0 \equiv (1,0,0)$, $\alpha=1$, and $T=5$.
For the effective field~\eqref{eq:llg:heff}, we choose $\ell_{\mathrm{ex}}=1$, a constant applied
field $\f \equiv (-2,-0.5,0)$, as well as an operator $\ppi$
which consists only of the stray field, i.e., $\ppi(\m) = \hstray(\m)$.

For the predictor step in {\tt PC1} and {\tt PC2}, respectively, we solve~\eqref{eq:pc1:implicit_pi}.
Since $\ppi_h$ effectively depends on $\v_h^\ell$ in~\eqref{eq:pc1:implicit_pi},
the linear system in the predictor step of Algorithm~\ref{alg:pc1:variational} and Algorithm~\ref{alg:pc2:variational} is solved with an inner fixed-point iteration which is stopped as soon as an accuracy of $10^{-10}$ (of $\norm[\L^2(\Omega)]{\v_h^i}$) is reached.
Other arising linear systems are solved with GMRES (or\ with CG for the hybrid FEM-BEM approach) with tolerance $10^{-12}$.
For the spatial discretization we consider a fixed triangulation $\TT_h$ of $\Omega$ generated by Netgen,
which consists of \num{3939} elements and \num{917} nodes (prescribed mesh size $h = 1/8$).

Since the exact solution of the problem is unknown, to compute the empirical convergence rates, we consider a reference solution $\m_{h,k_{\rm ref}}$ computed with the IMEX version of the second-order midpoint scheme from \cite{prs2017} using the above mesh and the time-step size $k_{\rm ref} = 2\cdot10^{-4}$.

Figure~\ref{fig:rates+time}\textrm{(a)} visualizes the experimental order of convergence of the four integrators. As expected, {\tt PC2} and {\tt PC2+IMEX} lead to second-order convergence in time.
Essentially, both integrators even lead quantitatively to the same accuracy of the numerical solution.
{\tt PC1} as well as {\tt PC1+IMEX} yield first-order convergence.
Differently from the classical $\theta$-method for linear second-order parabolic PDEs, due to the tangent plane constraint and the presence of the nodal projection, the {\tt PC1} integrator with $\theta=1/2$ (Crank--Nicolson-type) does not lead to any improvement of the convergence order in time (from first-order to second-order); see \cite{akst2014} for a formal analysis in the case of the tangent plane scheme.

In Figure~\ref{fig:rates+time}\textrm{(b)}, we plot the cumulative computational costs for the integration up to the final time $T$.
The computational effort improves considerably if the lower-order terms (i.e., the stray field) are integrated explicitly in time, since then the costly inner fixed-point iteration to solve~\eqref{eq:pc1:implicit_pi} is omitted.
Due to the more sophisticated corrector step in Algorithm~\ref{alg:pc2:variational} and Algorithm~\ref{alg:pc2:imex}, the second-order schemes \texttt{PC2} and \texttt{PC2+IMEX} are (slightly) more costly than their first-order counterparts \texttt{PC1} and \texttt{PC1+IMEX}, respectively.
\begin{figure}[h]
\centering
\begin{subfigure}{0.48\textwidth}
\begin{tikzpicture}
\pgfplotstableread{plots/H1_rates.dat}{\data}
\begin{loglogaxis}[
xlabel={time-step size $\times\, k_{\rm ref}$},
ylabel={Error},
xtick={0.0128, 0.0064, 0.0032, 0.0016, 0.0008, 0.0004},
xticklabels={64, 32, 16, 8, 4, 2},
height = 60mm,
legend style={
legend pos= south west},
ymax=0.5e-1,
ymin=0.5e-7,
ytick distance=10^1,
x dir=reverse
]
\addplot[only marks, red, mark=o, mark size=3.5, very thick] table[x=k, y=KW1] {\data};
\addplot[only marks, red, mark=x, mark size=3, very thick] table[x=k, y=KW1EE] {\data};
\addplot[only marks, teal, mark=o, mark size=3.5, very thick] table[x=k, y=KW2] {\data};
\addplot[only marks, teal, mark=x, mark size=3, very thick] table[x=k, y=KW2AB] {\data};
\addplot[black, dashed, ultra thick] table[x=k, y expr={(\thisrow{k})}]{\data};
\node at (axis cs:1e-3,1e-3) [anchor=north east] {\tiny $\mathcal{O}(k)$};
\addplot[black, dashed, ultra thick] table[x=k, y expr={(\thisrow{k}*\thisrow{k})*1.8}]{\data};
\node at (axis cs:1.2e-3,1.6e-6) [anchor=east] {\tiny $\mathcal{O}(k^2)$};
\legend{{\tt PC1},{\tt PC1+IMEX},{\tt PC2},{\tt PC2+IMEX}}
\end{loglogaxis}
\end{tikzpicture}
\caption{Error $\max_{\ell=0, \dots, L} \bignorm[\H^1(\Omega)]{\m_{h,k_{\rm ref}}^\ell - \m_{h,k}^\ell}$ for $k = 2^\ell \, k_{\rm ref}$ with $\ell  \in \{1,2,3,4,5,6\}$ and $k_{\rm ref} = 2\cdot10^{-4}$.}
\label{subfig:ratesH1_llg}
\end{subfigure}
\hfill 
\begin{subfigure}{0.48\textwidth}
\begin{tikzpicture}
\pgfplotstableread{plots/cumulative_isolated.dat}{\data}
\begin{axis}[
xlabel={physical time},
ylabel={comp. time [\si{\second}]},
height = 60mm,
legend style={
legend pos= north west},
xmax=5,
xmin=0,
ymin=0,
ymax=1.1e3,
ytick={200, 400, 600, 800, 1000},
]
\addplot[red,ultra thick] table[x=time, y=KW1] {\data};
\addplot[red, dashed, ultra thick] table[x=time, y=KW1EE] {\data};
\addplot[teal,ultra thick] table[x=time, y=KW2] {\data};
\addplot[teal, dashed, ultra thick] table[x=time, y=KW2AB] {\data};
\legend{{\tt PC1},{\tt PC1+IMEX},{\tt PC2},{\tt PC2+IMEX},} %
\end{axis}
\end{tikzpicture}
\caption{Cumulative computational time for $k = 8\cdot10^{-4}$.
Costs improve considerably for the IMEX versions.}
\label{subfig:cumulative}
\end{subfigure}
\caption{Experiments of Section~\ref{sec:convOrder}: Order of convergence (left) and cumulative computational time (right) of the integrators for $\theta = 1/2$.}
\label{fig:rates+time}
\end{figure}
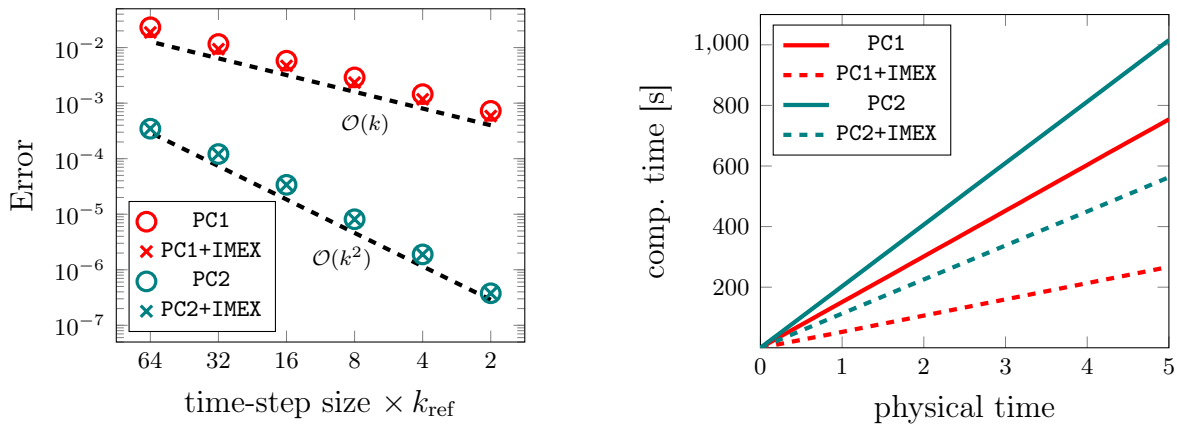

Further, we repeat the experiment for different values of $\theta \in [0,1]$ for both, \texttt{PC1+IMEX} and \texttt{PC2+IMEX}.
The results for \texttt{PC1+IMEX} in Figure~\ref{fig:order:pc2:theta}\textrm{(a)} confirm that the strong CFL condition $k = o(h^2)$, which is imposed to obtain stability and convergence of \texttt{PC1+IMEX} (see Remark~\ref{re:pc1:convergence}{\rm(ii)}) with $\theta < 1/2$, are also crucial in practice.
As expected, the observed order of convergence of \texttt{PC1+IMEX} is unaffected by the choice of $\theta \in [0, 1]$.

The results for \texttt{PC2+IMEX} shown in Figure~\ref{fig:order:pc2:theta}\textrm{(b)} are quite surprising:
While for $\theta \not= 1/2$, the simulation is not stable for larger time-step sizes $k > 0$, still second-order convergence is observed for all $0 \le \theta \le 1$ as the time-step size $k$ decreases below a certain threshold.
The preserved second-order accuracy for $\theta \not= 1/2$ might be a consequence of the degree of implicitness $\theta$ only appearing in the predictor, but not in the corrector of the scheme.
In contrast to stability for \texttt{PC1+IMEX}, the results of this experiment indicate that for stability of \texttt{PC2+IMEX} more restrictive CFL conditions are necessary for $\theta \not= 1/2$ than for $\theta = 1/2$.
This observation is further investigated in Section~\ref{sec:stability_pc2}.
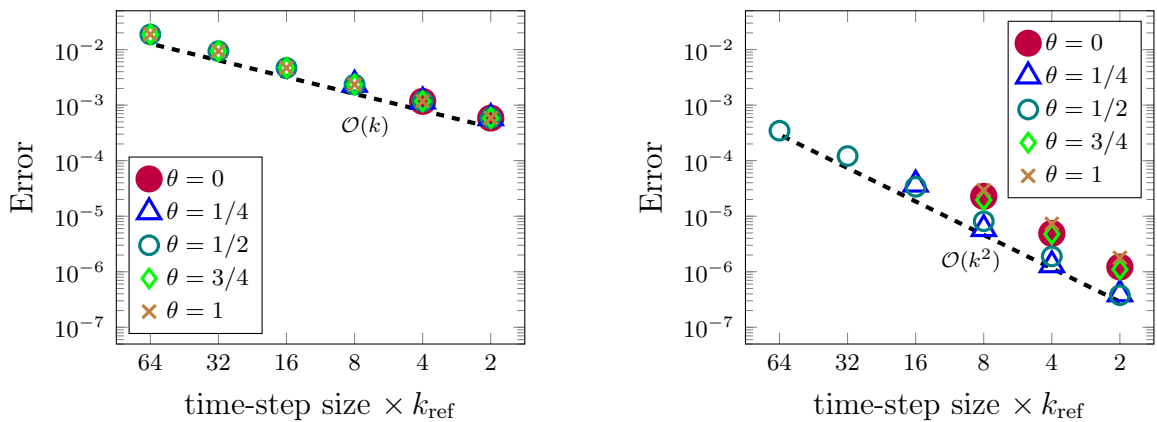
\begin{figure}[h]
\centering
\begin{subfigure}{0.48\textwidth}
\begin{tikzpicture}
\pgfplotstableread{plots/theta_H1_rates.dat}{\data}
\begin{loglogaxis}[
xlabel={time-step size $\times\, k_{\rm ref}$},
ylabel={Error},
xtick={0.0128, 0.0064, 0.0032, 0.0016, 0.0008, 0.0004},
xticklabels={64, 32, 16, 8, 4, 2},
height = 60mm,
legend style={
legend pos= south west},
ymax=0.5e-1,
ymin=0.5e-7,
ytick distance=10^1,
x dir=reverse
]
\addplot[only marks, purple, mark=*, mark size=4.5, very thick] table[x=k, y=KW1EE_t0_4] {\data};
\addplot[only marks, blue, mark=triangle, mark size=5, very thick] table[x=k, y=KW1EE_t1_4] {\data};
\addplot[only marks, teal, mark=o, mark size=3.5, very thick] table[x=k, y=KW1EE_t2_4] {\data};
\addplot[only marks, green, mark=diamond, mark size=3.5, very thick] table[x=k, y=KW1EE_t3_4] {\data};
\addplot[only marks, brown, mark=x, mark size=3.5, very thick] table[x=k, y=KW1EE_t4_4] {\data};
\addplot[black, dashed, ultra thick] table[x=k, y expr={(\thisrow{k})}]{\data};
\node at (axis cs:1e-3,1e-3) [anchor=north east] {\tiny $\mathcal{O}(k)$};
\legend{$\theta=0\phantom{/4}$,$\theta=1/4$,$\theta=1/2$,$\theta=3/4$,$\theta=1\phantom{/4}$}
\end{loglogaxis}
\end{tikzpicture}
\caption{Recomputation of Figure~\ref{fig:rates+time}{\rm(a)} for \texttt{PC1+IMEX} with various $\theta \in [0,1]$.}
\end{subfigure}
\hfill 
\begin{subfigure}{0.48\textwidth}
\begin{tikzpicture}
\pgfplotstableread{plots/theta_H1_rates.dat}{\data}
\begin{loglogaxis}[
xlabel={time-step size $\times\, k_{\rm ref}$},
ylabel={Error},
xtick={0.0128, 0.0064, 0.0032, 0.0016, 0.0008, 0.0004},
xticklabels={64, 32, 16, 8, 4, 2},
height = 60mm,
legend style={
legend pos= north east},
ymax=0.5e-1,
ymin=0.5e-7,
ytick distance=10^1,
x dir=reverse
]
\addplot[only marks, purple, mark=*, mark size=4.5, very thick] table[x=k, y=KW2AB_t0_4] {\data};
\addplot[only marks, blue, mark=triangle, mark size=5, very thick] table[x=k, y=KW2AB_t1_4] {\data};
\addplot[only marks, teal, mark=o, mark size=3.5, very thick] table[x=k, y=KW2AB_t2_4] {\data};
\addplot[only marks, green, mark=diamond, mark size=3.5, very thick] table[x=k, y=KW2AB_t3_4] {\data};
\addplot[only marks, brown, mark=x, mark size=3.5, very thick] table[x=k, y=KW2AB_t4_4] {\data};
\addplot[black, dashed, ultra thick] table[x=k, y expr={(\thisrow{k}*\thisrow{k})*1.8}]{\data};
\node at (axis cs:1.2e-3,1.6e-6) [anchor=east] {\tiny $\mathcal{O}(k^2)$};
\legend{$\theta=0\phantom{/4}$,$\theta=1/4$,$\theta=1/2$,$\theta=3/4$,$\theta=1\phantom{/4}$}
\end{loglogaxis}
\end{tikzpicture}
\caption{Recomputation of Figure~\ref{fig:rates+time}{\rm(a)} for \texttt{PC2+IMEX} with various $\theta \in [0,1]$.}
\end{subfigure}
\caption{Experiments of Section~\ref{sec:convOrder}: 
Order of convergence and stability for \texttt{PC1+IMEX} and \texttt{PC2+IMEX} for different values of $\theta \in [0, 1]$.
Stability is lost for \texttt{PC1+IMEX} (left) with $\theta = 0$ for $k \ge 8 \cdot k_{\operatorname{ref}}$, and with $\theta = 1/4$ for $k \ge 16 \cdot k_{\operatorname{ref}}$;
for \texttt{PC2+IMEX} (right) with $\theta \in \{0, 3/4, 1\}$ for $k \ge 16 \cdot k_{\operatorname{ref}}$, and with $\theta = 1/4$ for $k \ge 32 \cdot k_{\operatorname{ref}}$.}
\label{fig:order:pc2:theta}
\end{figure}

Overall, the proposed {\tt PC2+IMEX} integrator with $\theta = 1/2$ appears to be the method of choice with respect to experimental stability, computational time, and empirical accuracy. 

\subsection{Experimental stability of \texttt{PC2}}
\label{sec:stability_pc2}
We demonstrated the potential of (the IMEX version of) the second-order predictor-corrector scheme~\texttt{PC2} (\texttt{PC2+IMEX}) in Section~\ref{sec:mumag4} and Section~\ref{sec:convOrder}.
Our analysis guarantees unconditional well-posedness of the proposed second-order integrators in theory (Theorem~\ref{thm:pc2:well_posedness}) and in practice (Theorem~\ref{thm:pc2:inexact:well_posedness}).
However, neither the present work nor~\cite{kw2018} include a rigorous analysis on the stability of the second-order predictor-corrector scheme~\texttt{PC2} (Algorithm~\ref{alg:pc2:variational}), or its variant~\texttt{PC2+IMEX} (Algorithm~\ref{alg:pc2:imex}).
More precisely, it remains unclear whether the prescription of a CFL condition $k = o(h^\beta)$ for some $\beta > 0$ is sufficient to prove a discrete energy estimate of the form
\begin{align}
\norm{\Grad\m_h^J}^2 \le \norm{\Grad\m_h^0}^2 \qquad \text{for all } J=0, \dots, L\,,
\end{align}
where we omitted any lower-order contributions; see, e.g., \eqref{eq:pc1:imex:stability} for the full discrete energy estimate for~\texttt{PC1+IMEX}.

Hence, we close this section by a numerical study investigating the stability of~\texttt{PC2}.
Note that \texttt{PC2+IMEX} coincides with \texttt{PC2} for the exchange only case $\heff(\m) = \lex^2\Lapl\m$ of LLG, which is considered in the following experiments.
Motivated by the observations on stability of \texttt{PC2+IMEX} in Figure~\ref{fig:order:pc2:theta}{\rm(b)}, particular focus is put on the dependence on $0 \le \theta \le 1$, which controls the degree of implicitness in the predictor step~\eqref{eq:pc2:variational:predictor}.

\subsubsection{Setup}
\label{sec:numerics:pc2:stability:setup}
We consider the partition $\TT_h$ of the unit cube from Section~\ref{sec:convOrder}.
For a non-uniform initial condition $\m_h^0 \in \Mh$, we consider the exchange only case $\heff(\m) = \lex^2\Lapl\m$ of LLG and relax the dynamics until the (uniform) equilibrium state is reached.
Due to the absence of any lower-order contributions ($\ppi \equiv \0, \f \equiv \0$), the equilibrium state is a uniform magnetization in space, and the simulation is successfully stopped as soon as $\norm{\Grad\m_h^L}^2 \le 10^{-8}$ for some $L > 0$.
If $\norm{\Grad \m_h^{\ell+1}}^2 \le \norm{\Grad \m_h^{\ell}}^2$ for all $\ell=0, \dots, L-1$, the simulation is considered to be stable for the triangulation $\TT_h$ with fixed time-step size $k > 0$ and initial condition $\m_h^0 \in \Mh$.
If for some $\ell \ge 0$ the energy increases, i.e., if there holds $\norm{\Grad \m_h^{\ell+1}}^2 > \norm{\Grad \m_h^{\ell}}^2$, then we abort the simulation and we consider the simulation to be unstable for this combination of $\TT_h$,  $k > 0$, and $\m_h^0 \in \Mh$.

\subsubsection{Random initial state}
\label{sec:numerics:pc2:stability:random}
We choose the initial state $\m_h^0 \in \Mh$ such that $\{\m_{\z}(\z)\}_{\z \in \NN_h}$ is distributed randomly on $\sphere$.
\begin{figure}[h]
\centering
\begin{subfigure}{0.69\textwidth}
\begin{tikzpicture}
\pgfplotstableread{plots/cube_real_random_alpha1by1_stable.dat}{\stable}
\pgfplotstableread{plots/cube_real_random_alpha1by1_unstable.dat}{\unstable}
\begin{axis}[
xlabel={$\theta$},
x label style={at={(axis description cs:0.5,-0.15)},anchor=south},
ylabel={time-step size $k$ [$\cdot 10^{-3}$]},
y label style={at={(axis description cs:-0.0875,0)},anchor=west},
xtick={0, 0.1, 0.2, 0.3, 0.4, 0.5, 0.6, 0.7, 0.8, 0.9, 1},
xticklabels={0,,0.2,,0.4,,0.6,,0.8,,1},
ytick={1, 2, 3, 4, 5, 6, 7, 8, 9, 10, 11, 12},
yticklabels={1,,3,,5,,7,,9,,11,},
height = 60mm,
width = \textwidth,
legend style={
legend pos= north east},
ymin=0,
ymax=12.99,
xmin=-0.05,
xmax=1.05,
]
\addplot[only marks, green!50!black, mark=*, mark size=1, thick] table[x=theta, y=k] {\stable};
\addplot[only marks, red, mark=x, mark size=1.5, thick] table[x=theta, y=k] {\unstable};
\addplot [mark=none, very thick, dashed, gray] coordinates {(0.5, 0.5) (0.5, 12.5)};
\legend{stable, unstable, $\theta=1/2$}
\end{axis}
\end{tikzpicture}
\end{subfigure}
\hfill 
\begin{subfigure}{0.29\textwidth}
\begin{tikzpicture}
\node[] at (0, 0) {
\includegraphics[scale=0.08]{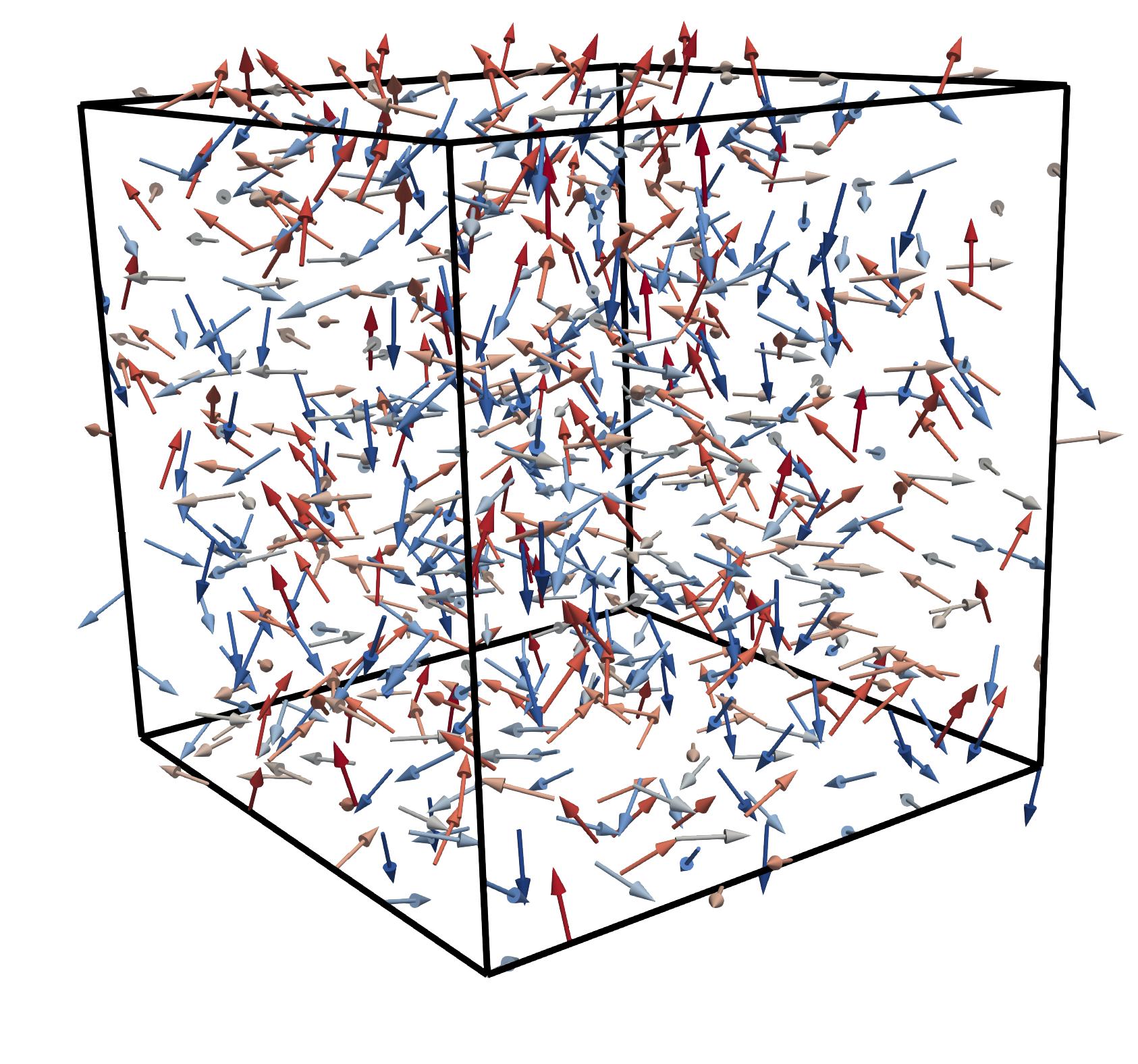}
};
\node[] at (0, -2.5) {$\m_h^0$};
\end{tikzpicture}
\end{subfigure}
\caption{Experiment of Section~\ref{sec:numerics:pc2:stability:random}.
Right: Random state $\m_h^0$ colored by the $z$-component; red pointing upwards, blue downwards.
Left: For all $\theta = 0/80, 1/80, \dots, 80/80$ and all $k = 1 \cdot 10^{-3}, 2 \cdot 10^{-3}, \dots, 12 \cdot 10^{-3}$, the stability of \texttt{PC2} is investigated.
}
\label{fig:cube_rand_stability}
\end{figure}
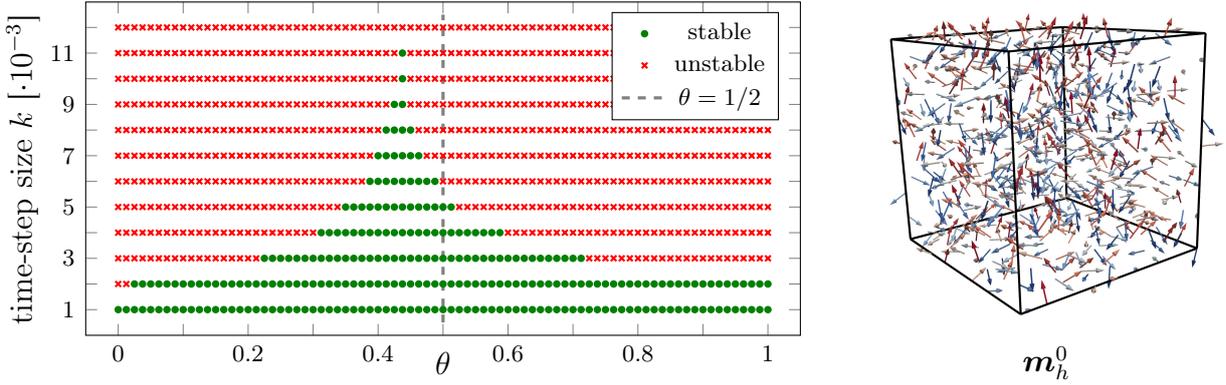

Figure~\ref{fig:cube_rand_stability} shows, that for any fixed $0 \le \theta \le 1$ the simulation is stable if the time-step size $k > 0$ is chosen small enough.
Clearly, stability of the simulation does not only depend on the chosen time-step size $k > 0$, but also on the parameter $\theta$:
Values of $\theta$ close to $1/2$ (best at $0.4375$ in this experiment) appear to be far less restrictive for the time-step size $k > 0$ than values farther from $1/2$.
We note that we repeated this experiment for various random initial states, all producing essentially the same result (not displayed).

\subsubsection{Hedgehog state}
\label{sec:numerics:pc2:stability:hedgehog}
We repeat the experiment from Section~\ref{sec:numerics:pc2:stability:random} for $\m_h^0$ being the so-called hedgehog state, i.e., considering the cube to be centered around the origin, for each vertex $\z \in \NN_h$ we set the initial value $\m_h^0(\z) := \z / |\z| \in \sphere$.
\begin{figure}[h]
\centering
\begin{subfigure}{0.69\textwidth}
\begin{tikzpicture}
\pgfplotstableread{plots/cube_hedgehog_stable.dat}{\stable}
\pgfplotstableread{plots/cube_hedgehog_unstable.dat}{\unstable}
\begin{axis}[
xlabel={$\theta$},
x label style={at={(axis description cs:0.5,-0.15)},anchor=south},
ylabel={time-step size $k$ [$\cdot 10^{-3}$]},
y label style={at={(axis description cs:-0.0875,0)},anchor=west},
xtick={0, 0.1, 0.2, 0.3, 0.4, 0.5, 0.6, 0.7, 0.8, 0.9, 1},
xticklabels={0,,0.2,,0.4,,0.6,,0.8,,1},
ytick={1, 7, 13, 19, 25},
height = 60mm,
width = \textwidth,
legend style={
legend pos= north east},
ymin=-1,
ymax=26.99,
xmin=-0.05,
xmax=1.05,
]
\addplot[only marks, green!50!black, mark=*, mark size=1, thick] table[x=theta, y=k] {\stable};
\addplot[only marks, red, mark=x, mark size=1.5, thick] table[x=theta, y=k] {\unstable};
\addplot [mark=none, very thick, dashed, gray] coordinates {(0.5, -0.0) (0.5, 26)};
\legend{stable, unstable, $\theta=1/2$}
\end{axis}
\end{tikzpicture}
\end{subfigure}
\hfill 
\begin{subfigure}{0.29\textwidth}
\begin{tikzpicture}
\node[] at (0, 0) {
\includegraphics[scale=0.08]{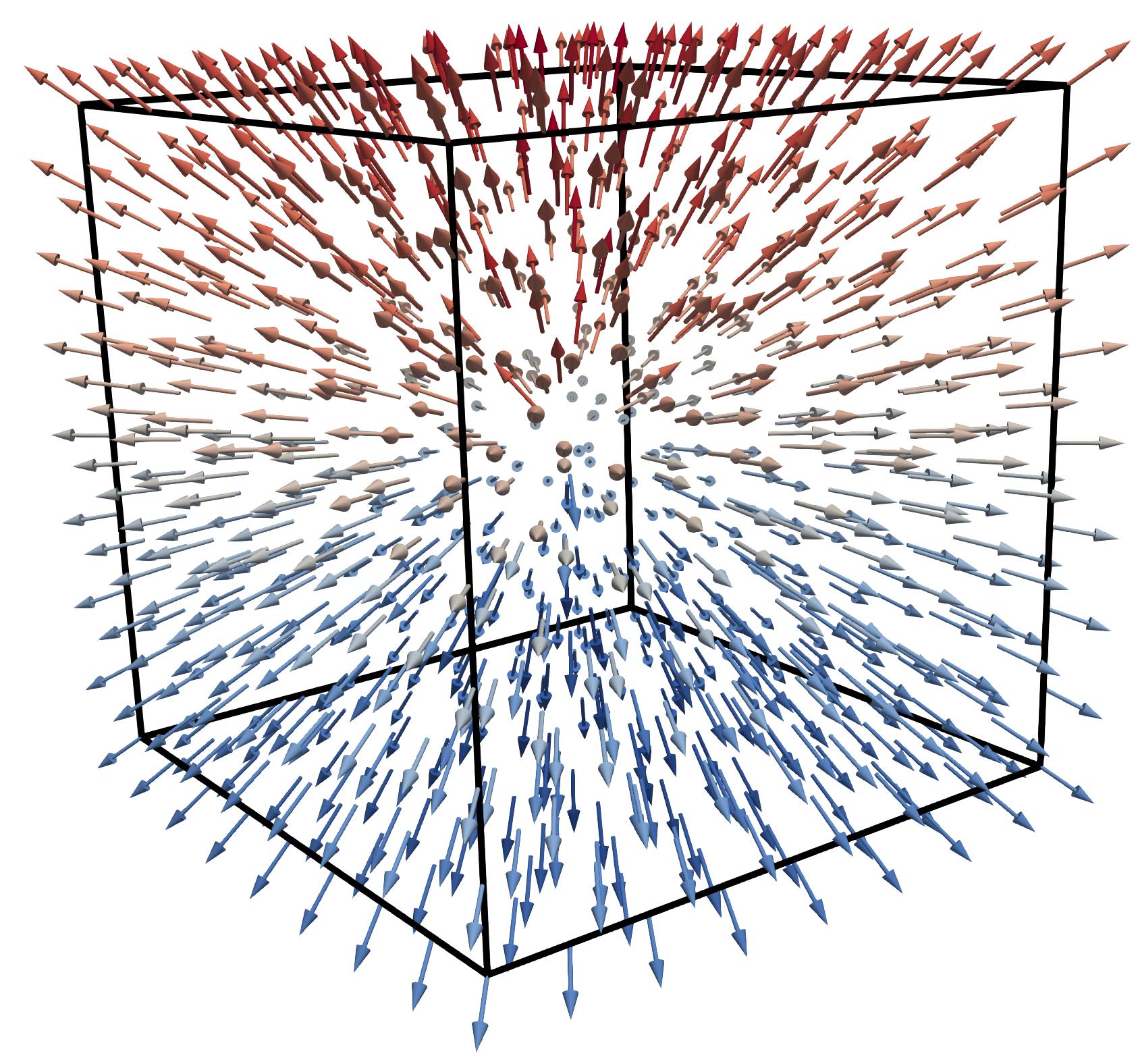}
};
\node[] at (0, -2.5) {$\m_h^0$};
\end{tikzpicture}
\end{subfigure}
\caption{Experiment of Section~\ref{sec:numerics:pc2:stability:hedgehog}:
Right: Hedgehog state $\m_h^0$ colored by the $z$-component; red pointing upwards, blue downwards.
Left: For all $\theta = 0/80, 1/80, \dots, 80/80$ and all $k = 1 \cdot 10^{-3}, 2 \cdot 10^{-3}, \dots, 25 \cdot 10^{-3}$, the stability of \texttt{PC2} is investigated.
}
\label{fig:cube_hedgehog_stability}
\end{figure}
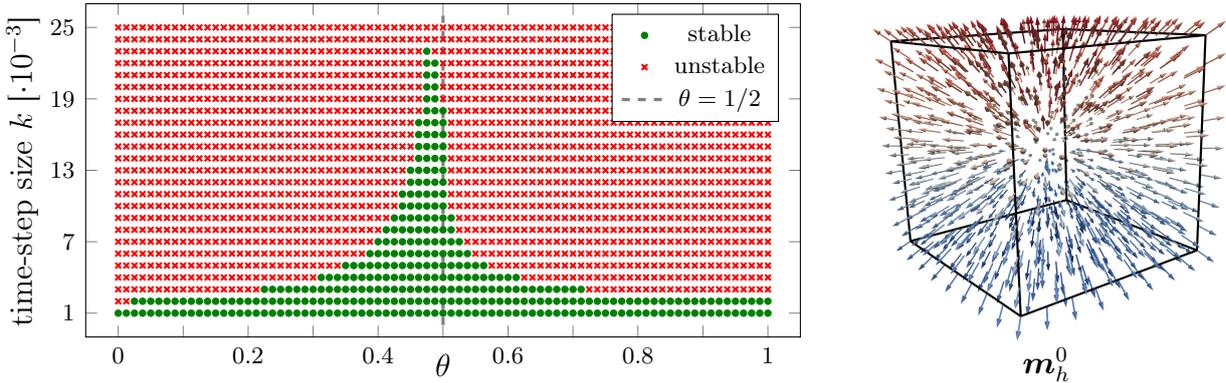

Figure~\ref{fig:cube_hedgehog_stability} shows, that again for any $0 \le \theta \le 1$ the simulation is stable if the time-step size $k > 0$ is chosen small enough.
As in Section~\ref{sec:numerics:pc2:stability:random}, values of $\theta$ close to $1/2$ appear to be far less restrictive for the time-step size $k > 0$ than values farther from $1/2$, with the optimal choice this time closer to $1/2$, precisely at $\theta = 0.475$.
Interestingly, for the parameter $\theta \in [0,1]$ chosen far from $1/2$, the results quantitatively match with those for the random initial state from Section~\ref{sec:numerics:pc2:stability:random}.
Closer to $1/2$, however, much larger time-step sizes $k > 0$ allow for stable simulations as for the random initial state.

\subsubsection{Variation of the Gilbert damping parameter}
\label{sec:numerics:pc2:stability:random:alpha}
We repeat the experiment from Section~\ref{sec:numerics:pc2:stability:random} for different values of $\alpha = 1/2, 1/4, 1/8, 1/16$.
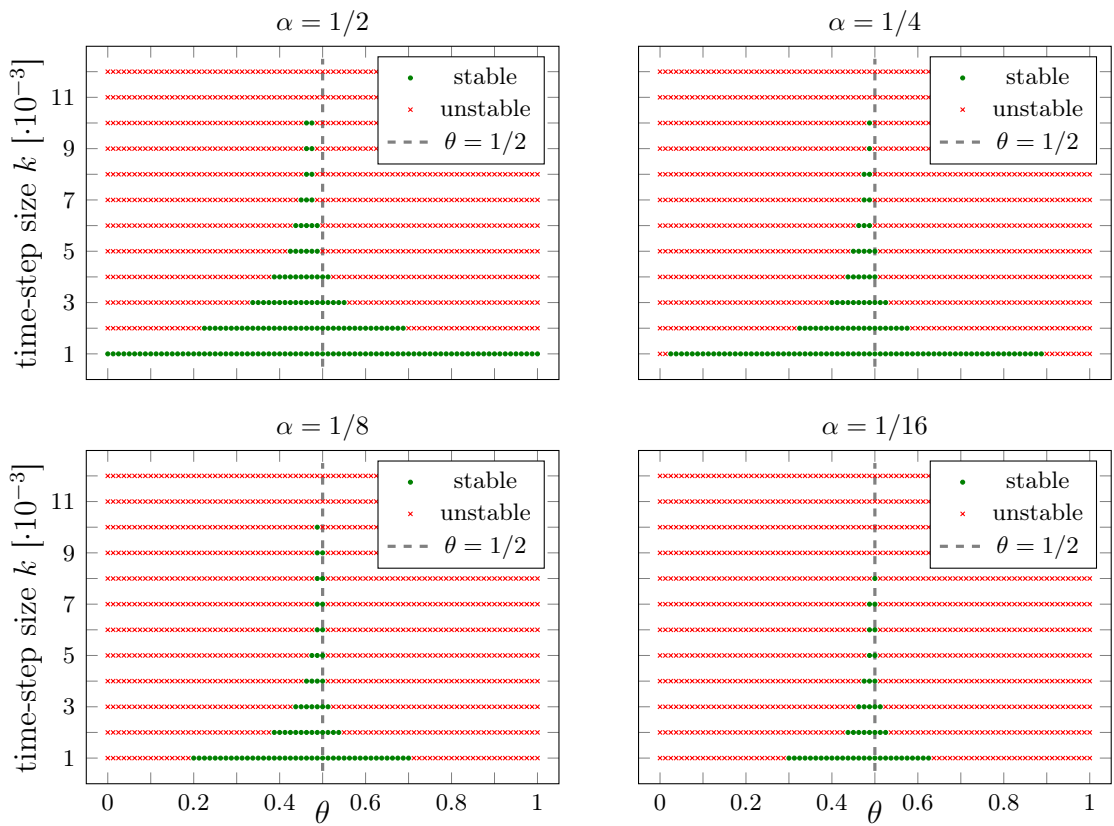
\begin{figure}[h]
\centering
\begin{subfigure}{0.49\textwidth}
\begin{tikzpicture}
\pgfplotstableread{plots/cube_real_random_alpha1by2_stable.dat}{\stable}
\pgfplotstableread{plots/cube_real_random_alpha1by2_unstable.dat}{\unstable}
\begin{axis}[
title={\footnotesize$\alpha=1/2$},
title style={at={(axis description cs:0.5,0.95)},anchor=south},
ylabel={time-step size $k$ [$\cdot 10^{-3}$]},
y label style={at={(axis description cs:-0.125,0)},anchor=west},
xtick={0, 0.1, 0.2, 0.3, 0.4, 0.5, 0.6, 0.7, 0.8, 0.9, 1},
xticklabels={,,,,,,,,,,},
ytick={1, 2, 3, 4, 5, 6, 7, 8, 9, 10, 11, 12},
yticklabels={1,,3,,5,,7,,9,,11,},
height = 60mm,
width = \textwidth,
legend style={
legend pos= north east},
ymin=0,
ymax=12.99,
xmin=-0.05,
xmax=1.05,
]
\addplot[only marks, green!50!black, mark=*, mark size=0.75] table[x=theta, y=k] {\stable};
\addplot[only marks, red, mark=x, mark size=1.125] table[x=theta, y=k] {\unstable};
\addplot [mark=none, very thick, dashed, gray] coordinates {(0.5, 0.5) (0.5, 12.5)};
\legend{stable, unstable, $\theta=1/2$}
\end{axis}
\end{tikzpicture}
\end{subfigure}
\hfill 
\begin{subfigure}{0.49\textwidth}
\begin{tikzpicture}
\pgfplotstableread{plots/cube_real_random_alpha1by4_stable.dat}{\stable}
\pgfplotstableread{plots/cube_real_random_alpha1by4_unstable.dat}{\unstable}
\begin{axis}[
title={\footnotesize$\alpha=1/4$},
title style={at={(axis description cs:0.5,0.95)},anchor=south},
xtick={0, 0.1, 0.2, 0.3, 0.4, 0.5, 0.6, 0.7, 0.8, 0.9, 1},
xticklabels={,,,,,,,,,,},
ytick={1, 2, 3, 4, 5, 6, 7, 8, 9, 10, 11, 12},
yticklabels={,,,,,,,,,,,,,,,,,},
height = 60mm,
width = \textwidth,
legend style={
legend pos= north east},
ymin=0,
ymax=12.99,
xmin=-0.05,
xmax=1.05,
]
\addplot[only marks, green!50!black, mark=*, mark size=0.75] table[x=theta, y=k] {\stable};
\addplot[only marks, red, mark=x, mark size=1.125] table[x=theta, y=k] {\unstable};
\addplot [mark=none, very thick, dashed, gray] coordinates {(0.5, 0.5) (0.5, 12.5)};
\legend{stable, unstable, $\theta=1/2$}
\end{axis}
\end{tikzpicture}
\end{subfigure}
\begin{subfigure}{0.49\textwidth}
\begin{tikzpicture}
\pgfplotstableread{plots/cube_real_random_alpha1by8_stable.dat}{\stable}
\pgfplotstableread{plots/cube_real_random_alpha1by8_unstable.dat}{\unstable}
\begin{axis}[
title={\footnotesize$\alpha=1/8$},
title style={at={(axis description cs:0.5,0.95)},anchor=south},
xlabel={$\theta$},
x label style={at={(axis description cs:0.5,-0.15)},anchor=south},
ylabel={time-step size $k$ [$\cdot 10^{-3}$]},
y label style={at={(axis description cs:-0.125,0.0)},anchor=west},
xtick={0, 0.1, 0.2, 0.3, 0.4, 0.5, 0.6, 0.7, 0.8, 0.9, 1},
xticklabels={0,,0.2,,0.4,,0.6,,0.8,,1},
ytick={1, 2, 3, 4, 5, 6, 7, 8, 9, 10, 11, 12},
yticklabels={1,,3,,5,,7,,9,,11,},
height = 60mm,
width = \textwidth,
legend style={
legend pos= north east},
ymin=0,
ymax=12.99,
xmin=-0.05,
xmax=1.05,
]
\addplot[only marks, green!50!black, mark=*, mark size=0.75] table[x=theta, y=k] {\stable};
\addplot[only marks, red, mark=x, mark size=1.125] table[x=theta, y=k] {\unstable};
\addplot [mark=none, very thick, dashed, gray] coordinates {(0.5, 0.5) (0.5, 12.5)};
\legend{stable, unstable, $\theta=1/2$}
\end{axis}
\end{tikzpicture}
\end{subfigure}
\hfill 
\begin{subfigure}{0.49\textwidth}
\begin{tikzpicture}
\pgfplotstableread{plots/cube_real_random_alpha1by16_stable.dat}{\stable}
\pgfplotstableread{plots/cube_real_random_alpha1by16_unstable.dat}{\unstable}
\begin{axis}[
title={\footnotesize$\alpha=1/16$},
title style={at={(axis description cs:0.5,0.95)},anchor=south},
xlabel={$\theta$},
x label style={at={(axis description cs:0.5,-0.15)},anchor=south},
xtick={0, 0.1, 0.2, 0.3, 0.4, 0.5, 0.6, 0.7, 0.8, 0.9, 1},
xticklabels={0,,0.2,,0.4,,0.6,,0.8,,1},
ytick={1, 2, 3, 4, 5, 6, 7, 8, 9, 10, 11, 12},
yticklabels={,,,,,,,,,,,,,,,,,},
height = 60mm,
width = \textwidth,
legend style={
legend pos= north east},
ymin=0,
ymax=12.99,
xmin=-0.05,
xmax=1.05,
]
\addplot[only marks, green!50!black, mark=*, mark size=0.75] table[x=theta, y=k] {\stable};
\addplot[only marks, red, mark=x, mark size=1.125] table[x=theta, y=k] {\unstable};
\addplot [mark=none, very thick, dashed, gray] coordinates {(0.5, 0.5) (0.5, 12.5)};
\legend{stable, unstable, $\theta=1/2$}
\end{axis}
\end{tikzpicture}
\end{subfigure}
\caption{Experiment of Section~\ref{sec:numerics:pc2:stability:random:alpha}:
With $\m_h^0$ the random state from Figure~\ref{fig:cube_rand_stability}(right) and different damping parameters $\alpha = 1/2, 1/4, 1/8, 1/16$, for all $\theta = 0/80, 1/80, \dots, 80/80$ and all $k = 1 \cdot 10^{-3}, 2 \cdot 10^{-3}, \dots, 12 \cdot 10^{-3}$, the stability of \texttt{PC2(+IMEX)} is investigated.
}
\label{fig:cube_rand_stability:alpha}
\end{figure}

Figure~\ref{fig:cube_rand_stability:alpha} shows that, if the damping parameter $\alpha$ decreases, smaller time-step sizes $k > 0$ are necessary to obtain stable simulations with \texttt{PC2(+IMEX)}.
This observation is in agreement with the role played by $\alpha$ in the model, i.e., incorporating dissipation.
Again, as previously observed for $\alpha = 1$, values of $\theta$ close to $1/2$ allow for larger time-step sizes $k > 0$ than values farther from $1/2$;
with the least restrictive choices at $\theta = 0.4375$ for $\alpha = 1$, $\theta = 0.4625$ to $0.475$ for $\alpha = 1/2$, $\theta = 0.4875$ for $\alpha = 1/4$, $\theta = 0.4875$ for $\theta = 1/8$, and $\theta = 0.5$ for $\alpha = 1/16$.
We obtain analogous results when varying $\alpha$ for the initial hedgehog state (not displayed).

\subsubsection{Concluding remarks on the stability of the second-order scheme}
All experiments in this section show that, in contrast to \texttt{PC1} (Theorem~\ref{thm:pc1:imex:stability}), larger values of $\theta$ do not improve stability of the second-order scheme \texttt{PC2}.
On the contrary, it is even the case that large values of $\theta$ perform as bad as small values of $\theta$.
For a generic simulation with \texttt{PC2(+IMEX)}, we suggest to pick the degree of implicitness $\theta = 1/2$ in the predictor.
Although, when considering one particular simulation setup, there might be better choices allowing for even larger time-step sizes, the choice $\theta = 1/2$ performed reliably throughout all experiments.
In particular, the results from Section~\ref{sec:numerics:pc2:stability:random:alpha} indicate that the deterioration of the optimal $\theta$ (with respect to stability) from $1/2$ might occur specifically for large values of $\alpha$, and quickly vanish as the damping parameter $\alpha$ decreases.
Moreover in future works, proving stability of \texttt{PC2} under some CFL condition for the special case $\theta = 1/2$ might be a possible first step in theoretically understanding stability of \texttt{PC2}.
This seems reasonable, as in this special case only the same highest-order term $\Lapl_h(\m_h^\ell + (k/2)\v_h^\ell)$ appears in the predictor and the corrector of \texttt{PC2}.
Hence these terms partially cancel out, when subtracting the two equations~\eqref{eq:pc2:variational:predictor}--\eqref{eq:pc2:variational:corrector} from each other.

\bibliographystyle{alpha}
\bibliography{ref}
\end{document}